\documentclass[12pt,leqno]{amsart}
\usepackage{amssymb, amsmath}

\usepackage[left=2cm,right=2cm,top=2.5cm,bottom=2.5cm]{geometry}

\usepackage{caption} 
\usepackage[colorlinks,linkcolor=blue,pagebackref=true,citecolor=red]{hyperref} 
\usepackage[all]{xy}
\usepackage{multicol}
\begin{document}

\theoremstyle{plain}
\newtheorem{thm}{Theorem}[section]
\newtheorem*{thm1}{Theorem 1}
\newtheorem*{thm2}{Theorem 2}
\newtheorem*{note*}{\it Note}
\newtheorem{lemma}[thm]{Lemma}
\newtheorem{lem}[thm]{Lemma}
\newtheorem{cor}[thm]{Corollary}
\newtheorem{prop}[thm]{Proposition}
\newtheorem{propose}[thm]{Proposition}
\newtheorem{variant}[thm]{Variant}
\theoremstyle{definition}
\newtheorem{notations}[thm]{Notations}
\newtheorem{rem}[thm]{Remark}
\newtheorem{rmk}[thm]{Remark}
\newtheorem{rmks}[thm]{Remarks}
\newtheorem{defn}[thm]{Definition}
\newtheorem{ex}[thm]{Example}
\newtheorem{exs}[thm]{Examples}
\newtheorem{claim}[thm]{Claim}
\newtheorem{ass}[thm]{Assumption}
\numberwithin{equation}{section}
\newcounter{elno}                
\def\points{\list
{\hss\llap{\upshape{(\roman{elno})}}}{\usecounter{elno}}} 
\let\endpoints=\endlist


\catcode`\@=11
%
%
\def\opn#1#2{\def#1{\mathop{\kern0pt\fam0#2}\nolimits}} 
\def\bold#1{{\bf #1}}%
\def\underrightarrow{\mathpalette\underrightarrow@}
\def\underrightarrow@#1#2{\vtop{\ialign{$##$\cr
 \hfil#1#2\hfil\cr\noalign{\nointerlineskip}%
 #1{-}\mkern-6mu\cleaders\hbox{$#1\mkern-2mu{-}\mkern-2mu$}\hfill
 \mkern-6mu{\to}\cr}}}
\let\underarrow\underrightarrow
\def\underleftarrow{\mathpalette\underleftarrow@}
\def\underleftarrow@#1#2{\vtop{\ialign{$##$\cr
 \hfil#1#2\hfil\cr\noalign{\nointerlineskip}#1{\leftarrow}\mkern-6mu
 \cleaders\hbox{$#1\mkern-2mu{-}\mkern-2mu$}\hfill
 \mkern-6mu{-}\cr}}}
%
%

%
\def\:{\colon}
\let\oldtilde=\tilde
\def\tilde#1{\mathchoice{\widetilde{#1}}{\widetilde{#1}}%
{\indextil{#1}}{\oldtilde{#1}}}
\def\indextil#1{\lower2pt\hbox{$\textstyle{\oldtilde{\raise2pt%
\hbox{$\scriptstyle{#1}$}}}$}}
\def\pnt{{\raise1.1pt\hbox{$\textstyle.$}}}
%

%
\let\amp@rs@nd@\relax
\newdimen\ex@\ex@.2326ex
\newdimen\bigaw@l
\newdimen\minaw@
\minaw@16.08739\ex@
\newdimen\minCDaw@
\minCDaw@2.5pc
\newif\ifCD@
\def\minCDarrowwidth#1{\minCDaw@#1}
\newenvironment{CD}{\@CD}{\@endCD}
\def\@CD{\def\A##1A##2A{\llap{$\vcenter{\hbox
 {$\scriptstyle##1$}}$}\Big\uparrow\rlap{$\vcenter{\hbox{%
$\scriptstyle##2$}}$}&&}%
\def\V##1V##2V{\llap{$\vcenter{\hbox
 {$\scriptstyle##1$}}$}\Big\downarrow\rlap{$\vcenter{\hbox{%
$\scriptstyle##2$}}$}&&}%
\def\={&\hskip.5em\mathrel
 {\vbox{\hrule width\minCDaw@\vskip3\ex@\hrule width
 \minCDaw@}}\hskip.5em&}%
\def\verteq{\Big\Vert&&}%
\def\noarr{&&}%
\def\vspace##1{\noalign{\vskip##1\relax}}\relax\let\amp@rs@nd@&\iffalse}\fi
 \CD@true\vcenter\bgroup\relax\let\\=\cr\iffalse}\fi\tabskip\z@skip\baselineskip20\ex@
 \lineskip3\ex@\lineskiplimit3\ex@\halign\bgroup
 &\hfill$\m@th##$\hfill\cr}
\def\@endCD{\cr\egroup\egroup}
%
\def\>#1>#2>{\amp@rs@nd@\setbox\z@\hbox{$\scriptstyle
 \;{#1}\;\;$}\setbox\@ne\hbox{$\scriptstyle\;{#2}\;\;$}\setbox\tw@
 \hbox{$#2$}\ifCD@
 \global\bigaw@\minCDaw@\else\global\bigaw@\minaw@\fi
 \ifdim\wd\z@>\bigaw@\global\bigaw@\wd\z@\fi
 \ifdim\wd\@ne>\bigaw@\global\bigaw@\wd\@ne\fi
 \ifCD@\hskip.5em\fi
 \ifdim\wd\tw@>\z@
 \mathrel{\mathop{\hbox to\bigaw@{\rightarrowfill}}\limits^{#1}_{#2}}\else
 \mathrel{\mathop{\hbox to\bigaw@{\rightarrowfill}}\limits^{#1}}\fi
 \ifCD@\hskip.5em\fi\amp@rs@nd@}
\def\<#1<#2<{\amp@rs@nd@\setbox\z@\hbox{$\scriptstyle
 \;\;{#1}\;$}\setbox\@ne\hbox{$\scriptstyle\;\;{#2}\;$}\setbox\tw@
 \hbox{$#2$}\ifCD@
 \global\bigaw@\minCDaw@\else\global\bigaw@\minaw@\fi
 \ifdim\wd\z@>\bigaw@\global\bigaw@\wd\z@\fi
 \ifdim\wd\@ne>\bigaw@\global\bigaw@\wd\@ne\fi
 \ifCD@\hskip.5em\fi
 \ifdim\wd\tw@>\z@
 \mathrel{\mathop{\hbox to\bigaw@{\leftarrowfill}}\limits^{#1}_{#2}}\else
 \mathrel{\mathop{\hbox to\bigaw@{\leftarrowfill}}\limits^{#1}}\fi
 \ifCD@\hskip.5em\fi\amp@rs@nd@}
%
%
\newenvironment{CDS}{\@CDS}{\@endCDS}
\def\@CDS{\def\A##1A##2A{\llap{$\vcenter{\hbox
 {$\scriptstyle##1$}}$}\Big\uparrow\rlap{$\vcenter{\hbox{%
$\scriptstyle##2$}}$}&}%
\def\V##1V##2V{\llap{$\vcenter{\hbox
 {$\scriptstyle##1$}}$}\Big\downarrow\rlap{$\vcenter{\hbox{%
$\scriptstyle##2$}}$}&}%
\def\={&\hskip.5em\mathrel
 {\vbox{\hrule width\minCDaw@\vskip3\ex@\hrule width
 \minCDaw@}}\hskip.5em&}
\def\verteq{\Big\Vert&}
\def\novarr{&}
\def\noharr{&&}
\def\SE##1E##2E{\slantedarrow(0,18)(4,-3){##1}{##2}&}
\def\SW##1W##2W{\slantedarrow(24,18)(-4,-3){##1}{##2}&}
\def\NE##1E##2E{\slantedarrow(0,0)(4,3){##1}{##2}&}
\def\NW##1W##2W{\slantedarrow(24,0)(-4,3){##1}{##2}&}
\def\slantedarrow(##1)(##2)##3##4{%
\thinlines\unitlength1pt\lower 6.5pt\hbox{\begin{picture}(24,18)%
\put(##1){\vector(##2){24}}%
\put(0,8){$\scriptstyle##3$}%
\put(20,8){$\scriptstyle##4$}%
\end{picture}}}
\def\vspace##1{\noalign{\vskip##1\relax}}\relax\let\amp@rs@nd@&\iffalse}\fi
 \CD@true\vcenter\bgroup\relax\let\\=\cr\iffalse}\fi\tabskip\z@skip\baselineskip20\ex@
 \lineskip3\ex@\lineskiplimit3\ex@\halign\bgroup
 &\hfill$\m@th##$\hfill\cr}
\def\@endCDS{\cr\egroup\egroup}
%
\newdimen\TriCDarrw@
\newif\ifTriV@
\newenvironment{TriCDV}{\@TriCDV}{\@endTriCD}
\newenvironment{TriCDA}{\@TriCDA}{\@endTriCD}
\def\@TriCDV{\TriV@true\def\TriCDpos@{6}\@TriCD}
\def\@TriCDA{\TriV@false\def\TriCDpos@{10}\@TriCD}
\def\@TriCD#1#2#3#4#5#6{%
\setbox0\hbox{$\ifTriV@#6\else#1\fi$}
\TriCDarrw@=\wd0 \advance\TriCDarrw@ 24pt
\advance\TriCDarrw@ -1em
\def\SE##1E##2E{\slantedarrow(0,18)(2,-3){##1}{##2}&}
\def\SW##1W##2W{\slantedarrow(12,18)(-2,-3){##1}{##2}&}
\def\NE##1E##2E{\slantedarrow(0,0)(2,3){##1}{##2}&}
\def\NW##1W##2W{\slantedarrow(12,0)(-2,3){##1}{##2}&}
\def\slantedarrow(##1)(##2)##3##4{\thinlines\unitlength1pt
\lower 6.5pt\hbox{\begin{picture}(12,18)%
\put(##1){\vector(##2){12}}%
\put(-4,\TriCDpos@){$\scriptstyle##3$}%
\put(12,\TriCDpos@){$\scriptstyle##4$}%
\end{picture}}}
\def\={\mathrel {\vbox{\hrule
   width\TriCDarrw@\vskip3\ex@\hrule width
   \TriCDarrw@}}}
\def\>##1>>{\setbox\z@\hbox{$\scriptstyle
 \;{##1}\;\;$}\global\bigaw@\TriCDarrw@
 \ifdim\wd\z@>\bigaw@\global\bigaw@\wd\z@\fi
 \hskip.5em
 \mathrel{\mathop{\hbox to \TriCDarrw@
{\rightarrowfill}}\limits^{##1}}
 \hskip.5em}
\def\<##1<<{\setbox\z@\hbox{$\scriptstyle
 \;{##1}\;\;$}\global\bigaw@\TriCDarrw@
 \ifdim\wd\z@>\bigaw@\global\bigaw@\wd\z@\fi
 \mathrel{\mathop{\hbox to\bigaw@{\leftarrowfill}}\limits^{##1}}
 }
 \CD@true\vcenter\bgroup\relax\let\\=\cr\iffalse}\fi
 \tabskip\z@skip\baselineskip20\ex@
 \lineskip3\ex@\lineskiplimit3\ex@
 \ifTriV@
 \halign\bgroup
 &\hfill$\m@th##$\hfill\cr
#1&\multispan3\hfill$#2$\hfill&#3\\
&#4&#5\\
&&#6\cr\egroup%
\else
 \halign\bgroup
 &\hfill$\m@th##$\hfill\cr
&&#1\\%
&#2&#3\\
#4&\multispan3\hfill$#5$\hfill&#6\cr\egroup
\fi}
\def\@endTriCD{\egroup} 

\newcommand{\mc}{\mathcal} 
\newcommand{\mb}{\mathbb} 
\newcommand{\surj}{\twoheadrightarrow} 
\newcommand{\inj}{\hookrightarrow} \newcommand{\zar}{{\rm zar}} 
\newcommand{\an}{{\rm an}} \newcommand{\red}{{\rm red}} 
\newcommand{\Rank}{{\rm rk}} \newcommand{\codim}{{\rm codim}} 
\newcommand{\rank}{{\rm rank}} \newcommand{\Ker}{{\rm Ker \ }} 
\newcommand{\Pic}{{\rm Pic}} \newcommand{\Div}{{\rm Div}} 
\newcommand{\Hom}{{\rm Hom}} \newcommand{\im}{{\rm im}} 
\newcommand{\Spec}{{\rm Spec \,}} \newcommand{\Sing}{{\rm Sing}} 
\newcommand{\sing}{{\rm sing}} \newcommand{\reg}{{\rm reg}} 
\newcommand{\Char}{{\rm char}} \newcommand{\Tr}{{\rm Tr}} 
\newcommand{\Gal}{{\rm Gal}} \newcommand{\Min}{{\rm Min \ }} 
\newcommand{\Max}{{\rm Max \ }} \newcommand{\Alb}{{\rm Alb}\,} 
\newcommand{\GL}{{\rm GL}\,} 
\newcommand{\ie}{{\it i.e.\/},\ } \newcommand{\niso}{\not\cong} 
\newcommand{\nin}{\not\in} 
\newcommand{\soplus}[1]{\stackrel{#1}{\oplus}} 
\newcommand{\oOplus}{\mbox{\fontsize{17.28}{21.6}\selectfont\( \oplus\)}}
\newcommand{\by}[1]{\stackrel{#1}{\rightarrow}} 
\newcommand{\longby}[1]{\stackrel{#1}{\longrightarrow}} 
\newcommand{\vlongby}[1]{\stackrel{#1}{\mbox{\large{$\longrightarrow$}}}} 
\newcommand{\ldownarrow}{\mbox{\Large{\Large{$\downarrow$}}}} 
\newcommand{\lsearrow}{\mbox{\Large{$\searrow$}}} 
\renewcommand{\d}{\stackrel{\mbox{\scriptsize{$\bullet$}}}{}} 
\newcommand{\dlog}{{\rm dlog}\,} 
\newcommand{\longto}{\longrightarrow} 
\newcommand{\vlongto}{\mbox{{\Large{$\longto$}}}} 
\newcommand{\limdir}[1]{{\displaystyle{\mathop{\rm lim}_{\buildrel\longrightarrow\over{#1}}}}\,} 
\newcommand{\liminv}[1]{{\displaystyle{\mathop{\rm lim}_{\buildrel\longleftarrow\over{#1}}}}\,} 
\newcommand{\norm}[1]{\mbox{$\parallel{#1}\parallel$}} 
\newcommand{\boxtensor}{{\Box\kern-9.03pt\raise1.42pt\hbox{$\times$}}} 
\newcommand{\into}{\hookrightarrow} \newcommand{\image}{{\rm image}\,} 
\newcommand{\Lie}{{\rm Lie}\,} 
\newcommand{\CM}{\rm CM}
\newcommand{\sext}{\mbox{${\mathcal E}xt\,$}} 
\newcommand{\shom}{\mbox{${\mathcal H}om\,$}} 
\newcommand{\coker}{{\rm coker}\,} 
\newcommand{\sm}{{\rm sm}} 
\newcommand{\tensor}{\otimes} 
\renewcommand{\iff}{\mbox{ $\Longleftrightarrow$ }} 
\newcommand{\supp}{{\rm supp}\,} 
\newcommand{\ext}[1]{\stackrel{#1}{\wedge}} 
\newcommand{\onto}{\mbox{$\,\>>>\hspace{-.5cm}\to\hspace{.15cm}$}} 
\newcommand{\propsubset} {\mbox{$\textstyle{ 
\subseteq_{\kern-5pt\raise-1pt\hbox{\mbox{\tiny{$/$}}}}}$}} 
\newcommand{\sB}{{\mathcal B}} \newcommand{\sC}{{\mathcal C}} 
\newcommand{\sD}{{\mathcal D}} \newcommand{\sE}{{\mathcal E}} 
\newcommand{\sF}{{\mathcal F}} \newcommand{\sG}{{\mathcal G}} 
\newcommand{\sH}{{\mathcal H}} \newcommand{\sI}{{\mathcal I}} 
\newcommand{\sJ}{{\mathcal J}} \newcommand{\sK}{{\mathcal K}} 
\newcommand{\sL}{{\mathcal L}} \newcommand{\sM}{{\mathcal M}} 
\newcommand{\sN}{{\mathcal N}} \newcommand{\sO}{{\mathcal O}} 
\newcommand{\sP}{{\mathcal P}} \newcommand{\sQ}{{\mathcal Q}} 
\newcommand{\sR}{{\mathcal R}} \newcommand{\sS}{{\mathcal S}} 
\newcommand{\sT}{{\mathcal T}} \newcommand{\sU}{{\mathcal U}} 
\newcommand{\sV}{{\mathcal V}} \newcommand{\sW}{{\mathcal W}} 
\newcommand{\sX}{{\mathcal X}} \newcommand{\sY}{{\mathcal Y}} 
\newcommand{\sZ}{{\mathcal Z}} \newcommand{\ccL}{\sL} 
 \newcommand{\A}{{\mathbb A}} \newcommand{\B}{{\mathbb 
B}} \newcommand{\C}{{\mathbb C}} \newcommand{\D}{{\mathbb D}} 
\newcommand{\E}{{\mathbb E}} \newcommand{\F}{{\mathbb F}} 
\newcommand{\G}{{\mathbb G}} \newcommand{\HH}{{\mathbb H}} 
\newcommand{\I}{{\mathbb I}} \newcommand{\J}{{\mathbb J}} 
\newcommand{\M}{{\mathbb M}} \newcommand{\N}{{\mathbb N}} 
\renewcommand{\P}{{\mathbb P}} \newcommand{\Q}{{\mathbb Q}} 

\newcommand{\R}{{\mathbb R}} \newcommand{\T}{{\mathbb T}} 
\newcommand{\U}{{\mathbb U}} \newcommand{\V}{{\mathbb V}} 
\newcommand{\W}{{\mathbb W}} \newcommand{\X}{{\mathbb X}} 
\newcommand{\Y}{{\mathbb Y}} \newcommand{\Z}{{\mathbb Z}} 

\title{Numerical characterizations for integral dependence of graded ideals}
\author{Suprajo Das}
\address{Department of Mathematics, Indian Institute of Technology Madras, Chennai, Tamil Nadu 600036, India}
\email{dassuprajo@gmail.com}
\author{Sudeshna Roy}
\address{Department of Mathematics, Indian Institute of Technology Gandhinagar, Palaj, Gandhinagar, Gujarat 382055, India}
\email{sudeshnaroy.11@gmail.com; sudeshna.roy@iitgn.ac.in}
\author{Vijaylaxmi Trivedi}
\address{Department of Mathematics, University at Buffalo (SUNY),  Buffalo, NY 14260}
\email{vija@math.tifr.res.in; vijaylax@buffalo.edu}

\begin{abstract}
Let $R=\oplus_{m\geq 0}R_m$ be a standard graded  equidimensional ring  over a field $R_0$ and $I\subseteq J$ be two non-nilpotent graded ideals in $R$. Then we give a set of numerical characterizations of the integral dependence  of $I$ and $J$ in terms of certain multiplicities. A novelty of this approach is that it does not involve localization and only requires checking computable and well-studied invariants.

In particular, we show the following: let $S=R[y]$, $\mathsf{I} = IS$ and $\mathsf{J} = JS$ and $\bf d$ be the maximum of the generating degrees of both $I$ and $J$. Let $c>{\bf d}$ be any given integer. Then
$$\overline{I} = \overline{J}\iff e\big(S[\mathsf{I}t]_{\Delta_{(c,1)}}\big) =  e\big(S[\mathsf{J}t]_{\Delta_{(c,1)}}\big),$$
where, for an ideal $I$ in $R$, the integer  $e\big(S[\mathsf{I}t]_{\Delta_{(c,1)}}\big)$ 
denotes the Hilbert-Samuel multiplicity of the standard graded domain $S[\mathsf{I}t]_{\Delta_{(c,1)}} = \oplus_{n\geq 0}(\mathsf{I}^n)_{cn}t^n$.
Further, if $I$ is of finite colength in $R$ then
$e\big(S[\mathsf{I}t]_{\Delta_{(c,1)}}\big) = 
c^de(R) - e(I,R)$.

If $R$ is a domain in addition then other numerical characterizations are the following:
\begin{align*}
\overline{I}  = \overline{J} & \iff \varepsilon(I)=\varepsilon(J)\;\;\mbox{and}\;\; e_i(R[It]) = e_i(R[Jt])\;\;\mbox{for all}\;\; 0\leq i <\dim(R/I),\\
& \iff 
 \varepsilon(I)=\varepsilon(J)\;\;\mbox{and}\;\;
 e\big(R[It]_{\Delta_{(c,1)}}\big) =  e\big(R[Jt]_{\Delta_{(c,1)}}\big)\;\;\mbox{for some integer}\;\; c > {\bf d},\\
& \iff  \varepsilon(\mathsf{I}_{\geq c})=\varepsilon(\mathsf{J}_{\geq c}),
\end{align*} 
where, for an ideal $I$ in $R$, $\varepsilon(I)$ and $\varepsilon(\mathsf{I}_{\geq c})$ denote the epsilon multiplicity of $I$ and $\mathsf{I}_{\geq c}$ respectively, and $e_i(R[It])$'s are the mixed multiplicities of the Rees algebra $R[It]$. The relation between $e_i(S[\mathsf{I}t])$ and the polar multiplicities of $\mathsf{I}_{\geq {\bf d}}$ provides another criterion in terms of polar multiplicities of $\mathsf{I}_{\geq {\bf d}}$.

The first two characterizations generalize the classical result of Rees for ideals of finite colengths.
Apart from several well-established results, the proofs of these results use the theory of density functions which was developed  in \cite{DRT24}.
\end{abstract}

\maketitle

\section{Introduction}

The main objective of this article is to provide a new (necessary and sufficient) numerical criterion for detecting integral dependence of arbitrary homogeneous ideals in a standard graded equidimensional ring  over a field. A novelty in our approach is that it does not involve localizations. This naturally came up as an application of various density functions, which the authors developed in their previous article \cite{DRT24}.

The idea of characterizing integral dependence through numerical invariants, was initiated in the pioneering work of Rees \cite{Ree61}. Since then finding such criteria became an important task in both commutative algebra and singularity theory. We shall now briefly trace the developments in this direction in chronological order. For a nice detailed survey, one may refer to \cite[Chapter 11]{HS06} or \cite[Introduction]{PTUV20}.

Let $(R, \mathbf{m})$ be a formally equidimensional Noetherian local ring. Rees \cite{Ree61} observed that two $\mathbf{m}$-primary ideals $I \subseteq J$ in $R$ have the same integral closure if and only if they have the same Hilbert-Samuel multiplicity. This result was applied by Teissier in his study of equisingularities, especially in his proof of Principle of Specialization of Integral Dependence (PSID), see \cite{Tei73}, \cite{Tei80a} and \cite{Tei80b}. B\"oger \cite{Bog69} extended Rees' result by showing that two ideals $I \subseteq J$ in $R$, where $I$ is equimultiple, have the same integral closure if and only if their localizations at every minimal prime over $I$ attain the same Hilbert-Samuel multiplicities. After a gap of three decades, Gaffney and Gasler \cite{GG99} extended the PSID for non $\mathbf{m}$-primary ideals by using the notion of Segre numbers. It allowed them to obtain generalizations of the earlier results of Rees and B\"oger in the complex analytic setting.

Since the numerical invariant Hilbert-Samuel multiplicity, which is used to characterize the integral dependence, can be  defined only for $\mathbf{m}$-primary ideals, mathematicians were led to define other notions of multiplicities that can be used to capture integral dependence. In \cite{AM93} Achilles and Manaresi introduced the concept of \emph{$j$-multiplicity} that is defined for any ideal, whereas the notion of \emph{$\varepsilon$-multiplicity} (also defined for any ideal) was introduced by Ulrich and Validashti \cite{UV11}. Flenner and Manaresi \cite{FM01} proved that if $I\subseteq J$ are arbitrary ideals in $R$, then they have the same integral closure if and only if they have the same $j$-multiplicity at every prime ideal where $I$ has maximal analytic spread. A similar statement using $\varepsilon$-multiplicity was given by Katz and Validashti in \cite{KV10}. However, both criteria require localization which may be difficult to verify in practice. There are also extensions of Rees’ theorem to the case of modules and algebras, see \cite{KT94}, \cite{Gaf03}, \cite{SUV01}, \cite{Cid24}, and \cite{CRPU24}.

Another important invariant is the so-called \emph{multiplicity sequence} of an ideal $I$ and it was defined by Achilles and Manaresi \cite{AM97}. The $j$-multiplicity of $I$ appears as one of the terms in its multiplicity sequence and the Segre numbers are a special case of the multiplicity sequence. Ciuperc\u{a} \cite{Ciu03} showed that if two ideals $I\subseteq J$ have the same integral closure then they have the same multiplicity sequence. Recently, Polini, Trung, Ulrich and Validashti \cite{PTUV20}  established a converse to this statement and thereby generalized Rees' theorem for arbitrary ideals.

In this paper  we give  simple criteria in the graded setup  for checking the integral dependence of two ideals $I\subseteq J$ in terms of various well-studied invariants. Here, we recall the notations used in the paper.

\subsection{Notations}\label{nnotation}
Let $k$ be a field and $R =\oplus_{m\geq 0}R_m$ be a standard graded finitely generated equidimensional algebra over $R_0=k$ of dimension $d\geq 2$. Let ${\bf m}=\oplus_{m\geq 1}R_m$ be the unique homogeneous maximal ideal of $R$. Let $I\subseteq J$ be two nonzero homogeneous ideals in $R$. Let $d(I)$ and $d(J)$ be the maximum generating degrees of $I$ and $J$ respectively, and set $\mathbf{d} = \max\{d(I), d(J)\}$. For any integer $c\geq {\bf d}$, let $I_{\geq c}=\oplus_{m\geq c}I_m$ and $ J_{\geq c}=\oplus_{m\geq c}J_m$ be the corresponding truncated ideals in $R$. Notice that $I_{\geq c} = I\cap {\bf m}^c$ and $J_{\geq c} = J\cap {\bf m}^c$. Let $$R[It]=\oOplus_{(m,n)\in\N^2}{(I^n)}_mt^n\quad \text{and}\quad R[Jt]=\oOplus_{(m,n)\in\N^2}{(J^n)}_mt^n$$ be the bigraded Rees algebras of $I$ and $J$ respectively. Further consider the $(c,1)$-diagonal subalgebras of $R[It]$ and $R[Jt]$ respectively, i.e., $$R[It]_{\Delta_{(c,1)}}  = \oOplus_{n\geq 0} {(I^{n})}_{cn}t^{n}\quad\mbox{and}\quad R[Jt]_{\Delta_{(c,1)}}  = \oOplus_{n\geq 0} {(J^{n})}_{cn}t^{n}.$$

Define $S=R[y]$, where $y$ is an indeterminate with $\deg y=1$, and $\mathbf{n} = \mathbf{m} + (y)$ be the unique homogeneous maximal ideal of $S$. Let $\mathsf{I} = IS$ and $\mathsf{J} = JS$ be the extensions of the ideals $I$ and $J$ in $S$ respectively. For $c\geq {\bf d}$, similarly define the truncated ideals $\mathsf{I}_{\geq c}=\oplus_{m\geq c}\mathsf{I}_m$ and $\mathsf{J}_{\geq c}=\oplus_{m\geq c}\mathsf{J}_m$ in $S$. We also consider the $(c,1)$-diagonal subalgebras $S[{\mathsf I}t]_{\Delta_{(c,1)}} = \oOplus_{n\geq 0}\left({\mathsf I}^{n}\right)_{cn}t^{n}$ and $S[{\mathsf J}t]_{\Delta_{(c,1)}} = \oOplus_{n\geq 0}\left({\mathsf J}^{n}\right)_{cn}t^{n}$ of the Rees algebras $S[{\mathsf I}t]$ and $S[{\mathsf J}t]$ respectively.

Further, recall that the Hilbert-Samuel multiplicity of a $d$-dimensional finitely generated graded $k$-algebra $A=\oplus_{m\geq 0}A_m$, is given by $e(A):=\lim_{m\to\infty}\tfrac{\ell_k(A_m)}{m^{d-1}/(d-1)!}$.

\begin{thm}[Theorem \ref{equidim}]\label{B1}
Adopt Notations \ref{nnotation}. Then
$$\overline{I} = \overline{J}\quad\text{if and only if}\quad e\big(S[\mathsf{I}t]_{\Delta_{(c,1)}}\big) =  e\big(S[\mathsf{J}t]_{\Delta_{(c,1)}}\big)\;\;\text{for some (every) integer}\;\; c>{\bf d}.$$
Further, if $I$ is of finite colength then
$$e\big(S[\mathsf{I}t]_{\Delta_{(c,1)}}\big) = 
c^de(R) - e(I,R)\quad\mbox{and}\quad
e\big(S[\mathsf{J}t]_{\Delta_{(c,1)}}\big) = 
c^de(R) - e(J,R).$$
\end{thm}

We give another numerical characterization which involves $\varepsilon$-multiplicity.

\begin{thm}[Theorem \ref{integral1}]\label{A} Adopt Notations \ref{nnotation} and further assume that $R$ is a domain. Then the following statements are true:
\begin{enumerate}
\item $\ell_R\left(\overline{J}/\overline{I}\right)<\infty$ if and only if $e_i(R[It]) = e_i(R[Jt])$ for all $0\leq i < \dim(R/I)$.
\item $\overline{J} = \overline{I}$ if and only if $\varepsilon(I)=\varepsilon(J)$ and $e_i(R[It]) = e_i(R[Jt])$ for all $0\leq i <\dim(R/I).$
\end{enumerate}
\end{thm}
Here $e_i(R[It])$ denotes the $i^{\mathrm{th}}$ RA-multiplicity of $R[It]$, see Notations \ref{n3}.

In \cite[Example 6.9]{CRPU24} the authors give an example which provides a negative answer to the following question, which they refer as  a folklore question: {\em Can polar multiplicities detect integral dependence?} Their example is in the graded setting. However, they show in \cite[Corollary 6.15]{CRPU24} that the polar multiplicities do detect the integral dependence of $I\subseteq J$ provided $I$ and $J$ are equigenerated ideals generated by same degree elements.

In the light of Remark \ref{polarmultiplicity}, we reconcile with their statements as follows. Here Theorem \ref{A} asserts that the polar multiplicities of $(R, I_{\geq \beta})$ and $(R,J_{\geq\beta})$ coincide if and only if $\ell_R({\overline J}/{\overline I})<\infty$, where $\beta\geq {\bf d}$ is any integer. Further, the polar multiplicities of $(R, I_{\geq \beta})$ and $(R,J_{\geq\beta})$ together with the epsilon multiplicities of $I$ and $J$ characterize the integral dependence of $I$ and $J$. Now when $I$ and $J$ are equigenerated and $\ell_R({\overline J}/{\overline I})< \infty$ then, by Corollary \ref{equigen}, the equality $\varepsilon(I) = \varepsilon(J)$ holds if and only if $I$ and $J$ are generated by same degree generators.

In general (when $I$ and $J$ are not necessarily equigenerated) Theorem \ref{B} implies that the  polar multiplicities of $(S,\mathsf{I_{\geq\beta}})$ and $(S,\mathsf{J_{\geq\beta}})$ characterize the integral dependence of $I$ and $J$. It also asserts that the $\varepsilon$-multiplicities (or $j$-multiplicities) of $(S,\mathsf{I_{>\beta}})$ and $(S,\mathsf{J_{>\beta}})$ are sufficient to characterize the integral dependence of $I$ and $J$.

We note that both  Theorem \ref{B1} and Theorem \ref{A} give a natural generalization of the classical theorem of Rees \cite{Ree61} in the graded situation. We shall now comment on the computational aspects of Theorem \ref{A}.

It is shown in \cite{DDRV24} that there exists a relationship  between multiplicities $e_i(R[It])$ of the Rees algebra $R[It]$ and mixed multiplicities of the pair of ideals $({\bf m}, I)$. The mixed multiplicities of a given  pair of ideals can be computed using the Macaulay2 package {\sf MixedMultiplicity} presented in \cite{GMRV23}. This makes criterion $(1)$ of Theorem \ref{A} computable. On the other hand, verifying criteria $(2)$ of Theorem \ref{A} can be hard. This is because $\varepsilon$-multiplicity is a difficult invariant to compute (in fact, it can take irrational values). Though, in many interesting situations, there are some methods available to compute the $\varepsilon$-multiplicities, see \cite{DDRV24}, \cite{JM13}, and \cite{JMV15}.

Following  theorem  gives a characterization of integral dependence in terms of well-known invariants which are amenable to computations, where $e_{i}\big(\mathbf{n} | \mathsf{I}_{\geq \bf d} \big)$ denote the $i^{\mathrm{th}}$ mixed multiplicity for the pair $({\mathbf{n}},\mathsf{I}_{\geq\bf d})$, see Section \ref{mixmul}. In Section \ref{M2section} we write a Macaulay2 script based on the criteria described in the theorem to check the equality of the integral closures.

\begin{thm}[Theorem \ref{comp}]\label{B}
Adopt Notations \ref{nnotation} and further assume that $R$ is a domain. Then the following  statements are equivalent:
\begin{enumerate}
\item[$(i)$] $\overline{I} = \overline{J}$.
\item[$(ii)$] $e_i\left(S[\mathsf{I}t]\right) = e_i\left(S[\mathsf{J}t]\right)$ for all $i$, where $0\leq i\leq \dim R/I$.
\item[$(iii)$] $e_{i}\big(\mathbf{n} |\mathsf{I}_{\geq\bf d}\big)=e_{i}\big(\mathbf{n} |\mathsf{J}_{\geq\bf d}\big)$ for all $i$, where $0\leq d-i \leq \dim\;R/I$.
\item[$(iv)$] $\varepsilon\big(\mathsf{I}_{\geq c}\big) = \varepsilon\big(\mathsf{J}_{\geq c}\big)$ for some integer (all integers) $c>{\bf d}$.
\item[$(v)$] $j\big(\mathsf{I}_{\geq c}\big) = j\big(\mathsf{J}_{\geq c}\big)$ for some integer (all integers) $c>{\bf d}$.
\end{enumerate}
\end{thm}

We now give a brief outline of the methods used in the proofs. In \cite{DRT24} we have established the following: given a homogeneous ideal $I$ in $R$ one may associate three real-valued nonnegative density functions (see Section \ref{preliminaries}), namely the {\em adic density function} $f_{\{I^n\}}$, the {\em saturation density function} $f_{\{\widetilde {I^n}\}}$, and the {\em $\varepsilon$-density function} $f_{\varepsilon(I)}$. All these functions are continuous almost everywhere (in fact, the number of points of discontinuity is at the most two). Further, $f_{\varepsilon(I)} = f_{\{\widetilde {I^n}\}} - f_{\{I^n\}}$ and $f_{\{\widetilde {I^n}\}}(x) = f_{\{I^n\}}(x)$, for all $x> d(I)$. In particular, $f_{\varepsilon(I)}$ is compactly supported and $\int_{\R} f_{\varepsilon(I)} = \varepsilon(I)$. We showed that the three density functions do not change if we replace $I$ by its integral closure ${\overline I}$.

In order to give a numerical characterization for  ${\overline{I}} = \overline{J}$, we first provide a numerical characterization for the weaker assertion, namely the finiteness of $\ell_R({\overline{J}}/{\overline{I}})$. Here we see that this finiteness condition is equivalent to the equality of the corresponding saturation density functions, where we know that the saturation density functions are nothing but restriction of the volume functions to appropriate line segments in the real N\'eron-Severi space of $\mathrm{Proj}\; R$. So one way implication is straightforward as the saturation density function for an ideal ${\overline{I}}$ can be realised in terms of the associated ideal sheaf ${\overline{\sI}}$ on $\mathrm{Proj}\;R$.

However, the heart of the argument lies in establishing the converse. In fact, a much stronger statement holds, which is, if the saturation density functions for $I$ and $J$ agree at some (every) integer $c>{\bf d}$ then $\ell_R(\overline{J}/\overline{I}) <\infty$. For this we provide two different proofs: $(1)$ an algebraic proof (Proposition \ref{sing_pt_adic}) based on a result from \cite[Theorem 3.3]{SUV01}, and $(2)$ a geometric proof (Proposition \ref{impsat1}) using a result on the equality of {\em volume functions} from \cite[Theorem A]{FKL16}. However, for the geometric proof we require the extra hypothesis that the underlying field is perfect, but then we are allowed to choose $c>{\bf d}$ to be any  real number (which gives more options to choose the diagonal subalgebras). On the other hand, we know that the equality of the saturation density functions at the point $x=c$ is same as the equality  $e\big(R[It]_{\Delta_{(c,1)}}\big) = e\big(R[Jt]_{\Delta_{(c,1)}}\big)$. Note that both $R[It]_{\Delta_{(c,1)}}$ and $R[Jt]_{\Delta_{(c,1)}}$ are $d$-dimensional standard graded Noetherian domains over $R_0$. Thus in Theorem \ref{impsat}, using saturation density functions, we give the following numerical characterization of finiteness of $\ell_R({\overline{J}}/{\overline{I}})$.

\begin{enumerate}
\item $\ell_R({\overline{J}}/{\overline{I}}) <\infty$.
\item $f_{\{\widetilde{I^n}\}}(x) = f_{\{\widetilde{J^n}\}}(x)$ for all $x\geq 0$.
\item $e\big(R[It]_{\Delta_{(c,1)}}\big) = e\big(R[Jt]_{\Delta_{(c,1)}}\big)$ for some integer $c>{\bf d}$.
\end{enumerate}

We first note that the equality $\overline{I} = \overline{J}$ holds if and only $\ell_S(\overline{\mathsf{J}}/\overline{\mathsf{I}})
<\infty$. Now, by construction, for a given graded ideal $I$, the adic density function $f_{\{I^n\}}$ is a pointwise limit of integrable step functions, and
therefore using the Lebesgue's dominated convergence theorem, we can express in (\ref{satdensity}) the saturated density function
$f_{\{\widetilde{\mathsf{I^n}}\}}$ as integrals of the adic density function $f_{\{I^n\}}$. On the other hand for ideals  $I\subseteq J$ the function
$f_{\{J^n\}}-f_{\{I^n\}}$ is a nonnegative function which is continuous at almost all points. Therefore by (\ref{satdensity}) we deduce that the equality of the
adic density functions is same as the equality of the saturation density functions for the respective extended ideals.
Thus, in Theorem~\ref{adicfunction} we establish the equivalence of the following statements:
\begin{enumerate}
\item ${\overline{I}} = {\overline{J}}$.
\item $f_{\{I^n\}}(x) = f_{\{J^n\}}(x)$ for all $x\geq 0$.
\item $f_{\{\widetilde{\mathsf{I}^n}\}}(x) = f_{\{\widetilde{\mathsf{J}^n}\}}(x)$ for all $x\geq 0$.
\item $e\big(S[\mathsf{I}t]_{\Delta_{(c,1)}}\big) = e\big(S[\mathsf{J}t]_{\Delta_{(c,1)}}\big)$ for some (every) integer $c > {\bf d}$.
\end{enumerate}

It was previously shown in \cite[Corollary 6.1]{DRT24} that $f_{\{I^n\}} = f_{\{J^n\}}$ is equivalent to the assertion $e\big(R[It]_{\Delta_{(p,q)}}\big) = e\big([Jt]_{\Delta_{(p,q)}}\big)$, for all $p, q \in \N$.

Since the numerical characterizations are in terms of well-studied invariants, we expect that the techniques and results of this article will have applications from both theoretical and computational points of view. In future, we would like to approach such results for local rings by extending the notion of density functions in the local setup.

\section{Preliminaries}\label{preliminaries}
\subsection{Some results on density functions} In \cite{DRT24} we have established the existence of density functions to study the asymptotic growth of the ideals arising from the  powers of graded ideals. Here we recall the definitions and properties of the three density functions, namely,
$(1)$ $I$-adic density function,
$(2)$ saturation density function and
$(3)$ epsilon density function.
  
\begin{notations}\label{n1}
Let $R$, $\bf m$, and $k$ be as in Notations \ref{nnotation}. Further assume that $R$ is a domain. Let $I$ be a nonzero homogeneous ideal in $R$. For a given choice of homogeneous generators of $I$, let the set of degrees of its generators be $d_1,\ldots, d_l$. By reindexing  we may assume that $d_1<\cdots <d_l$ and denote $d(I) = d_l$. The \emph{saturation} of $I$ is defined to be the graded ideal ${\widetilde{I}} = I\colon_R{\bf m}^{\infty} = \{f\in R\mid f\cdot\mathbf{m}^c \subset I\;\text{for some}\;c\in\mathbb{N}\}$.
\end{notations}

\begin{defn}\label{d1}
$\quad$
\begin{enumerate}
 \item Let $f_n\colon\R_{\geq 0}\to \R_{\geq 0}$ be the function given by
 $$f_n(x) = \dfrac{\ell_k \big(\left(I^n\right)_{\lfloor xn\rfloor}\big)}{n^{d-1}/d!}.$$
The \emph{adic density function} of $I$ is the function $f_{\{I^n\}} \colon \mathbb{R}_{\geq 0}\to \mathbb{R}_{\geq 0}$ defined by
$$f_{\{I^n\}}(x) = \limsup_{n\to\infty}
f_n(x).$$
\item Let ${g_n}\colon\R_{\geq 0}\to \R_{\geq 0}$ be the function given by
$$g_n(x) = \dfrac{\ell_k \big(\left(I^n\colon_R \mathbf{m}^{\infty}\right)_{\lfloor xn\rfloor}\big)}{n^{d-1}/d!}.$$
The \emph{saturation density function} of $I$ is the function $f_{\{\widetilde{I^n}\}} \colon \mathbb{R}_{\geq 0}\to \mathbb{R}_{\geq 0}$ defined by
$$f_{\{\widetilde{I^n}\}}(x) = \limsup_{n\to\infty}{g_n}(x).$$
\item Let ${f_n(\varepsilon)}:\R_{\geq 0}\to \R_{\geq 0}$ be the function given by
$$f_n(\varepsilon)(x) = \frac{\ell_k\big(\left(H^0_{\bf m}(R/I^n)\right)_{\lfloor xn \rfloor}\big)}{n^{d-1}/d!}.$$
The \emph{epsilon density ($\varepsilon$-density) function} of $I$ is the function $f_{\varepsilon(I)}\colon \mathbb{R}_{\geq 0}\to \mathbb{R}_{\geq 0}$ defined by
$$f_{\varepsilon(I)}(x) = \limsup_{n\to\infty} f_n(\varepsilon)(x).$$
\end{enumerate}
\end{defn}

We shall now list the properties of these  density functions.

\begin{thm}\label{2.3}\cite[Theorem $4.4$]{DRT24} Let $R$ and $I$ be as in Notations \ref{n1}. Then the following statements are true.
\begin{enumerate}
 \item[$(i)$] The sequence $\{f_n\}_{n\in \N}$ converges (locally uniformly) to $f_{\{I^n\}}$ on the set $\R_{\geq 0}\setminus \{d_1\}$.
 \item[$(ii)$] Further, $$f_{\{I^n\}}(x) = \begin{cases}
                                      0 & \forall x\in [0, d_1),\\
                                      {\bf p}_1(x) & \forall x\in (d_1, d_2],\\
                                      {\bf p}_j(x) & \forall x\in [d_j, d_{j+1}]~~\mbox{and}~~j=2, \ldots, l,
                                     \end{cases}$$
  where $d_{l+1} = \infty$, and for each $j=1,\ldots, l$, ${\bf p}_j(x)$ is a nonzero polynomial of degree $\leq d-1$ with rational coefficients. Moreover, ${\bf p}_l(x)$ has degree $d-1$.
  \item[$(iii)$] For any real number  $c>0$, we have $$\int_0^c f_{\{I^n\}}(x)dx = \lim_{n\to \infty} \frac{\sum_{m=0}^{\lfloor cn\rfloor}\ell_k\left(\left(I^n\right)_m\right)}{n^d/d!}.$$
\end{enumerate}
\end{thm}
 
Next we recall the relevant notations to list the properties of the saturation density functions. Consider the projective variety $V = \mathrm{Proj}\;R$ with a very ample invertible sheaf $\mathcal{O}_V(1)$. Let $\mathcal{I}$ be the ideal sheaf associated to $I$ on $V$. Let $$\pi \colon X =  {\mathbf{Proj}} \left(\oplus_{n\geq 0}\mathcal{I}^n\right) \longto V$$ be the blow up of $V$ along $\mathcal{I}$. Then $\mathcal{I}$ becomes locally principal on $X$, i.e., there is an effective Cartier divisor $E$ on $X$ (namely the exceptional divisor of $\pi$) such that $\mathcal{I}\mathcal{O}_X = \mathcal{O}_{X}(-E)$. Let $H$ be the pullback of a hyperplane section on $V$. Define the constants
$$\alpha_{I} = \min\{x\in \mathbb{R}_{\geq 0}\mid xH-E~~~~\mbox{is pseudoeffective}\}\quad \text{and}\quad
   \beta_I = \min\{x\in \mathbb{R}_{\geq 0}\mid xH-E~~~~\mbox{is nef}\},
$$
where we have $0\leq \alpha_I \leq d_1$ and $\beta_I \leq d_l=d(I)$. If $I$ is not an ideal of finite colength then $E$ is an effective divisor of positive degree which implies that $\alpha_I >0$.

In view of the results in \cite[Section 5 and Remark 5.22]{DRT24}, we get the following theorem.

\begin{thm}\label{mainsat}
Let $R$ and $I$ be as in Notations \ref{n1}. Then the following statements are true:
 \begin{enumerate}
 \item[$(i)$] The sequence $\{g_n\}_{n\in \N}$ converges (locally uniformly) to $f_{\{\widetilde{I^n}\}}$ on the set $\R_{\geq 0}$. Furthermore,
 $$f_{\{\widetilde{I^n}\}}(x) = d\cdot\mathrm{vol}_X(xH-E).$$
 \item[$(ii)$] The function $f_{\{\widetilde{I^n}\}} \colon \mathbb{R}_{\geq 0} \to \mathbb{R}_{\geq 0}$ is continuous.
 \item[$(iii)$] We have $f_{\{\widetilde{I^n}\}}(x) = 0$ for all $x \leq \alpha_I$, and $f_{\{\widetilde{I^n}\}}$ is a strictly increasing continuously differentiable function on the interval $(\alpha_I, \infty)$.
 \item[$(iv)$] If $x\geq \beta_I$, then $$f_{\{\widetilde{I^n}\}}(x) = d\cdot(xH-E)^{d-1} = \sum_{i=0}^{d-1}(-1)^{d-1-i}\dfrac{d!}{(d-i-1)!i!}(H^{i}\cdot E^{d-1-i})x^{i},$$ where $H^{d-1-i}\cdot E^i$ denotes the intersection number of the Cartier divisors $H^{d-1-i}$ and $E^i$.
 \item[$(v)$] For any real number  $c>0$, $$\int_0^c f_{\{\widetilde{I^n}\}}(x)dx = \lim_{n\to \infty} \frac{\sum_{m=0}^{\lfloor cn\rfloor}\ell_k\big(\big(\widetilde{I^n}\big)_m\big)}{n^d/d!}.$$
 \item[$(vi)$] We have $f_{\{\widetilde{I^n}\}}(x) = f_{\{I^n\}}(x)$ for all $x>d(I)$.
 \item[$(vii)$] If $I$ is of finite colength then $f_{\{\widetilde{I^n}\}}(x) = d\cdot e(R)x^{d-1}$, where
 $e(R)$ denotes the Hilbert-Samuel multiplicity of $R$.
 Here, $\alpha_I = \beta_I = 0$.
 \end{enumerate}
 \end{thm}

Using the above two density functions one can prove the properties of the $\varepsilon$-density function leading to the invariant $\varepsilon$-multiplicity, see \cite[Theorems 5.5 and 5.19]{DRT24}.

We recall the notion of epsilon multiplicity, which  was introduced in the works of Kleiman, Ulrich and Validashti, see \cite{KUV} and \cite{UV08}, and defined as follows: if $(R, {\bf m})$ is a $d$-dimensional Noetherian local ring and $I$ is an ideal in $R$, then the \emph{epsilon multiplicity} $\varepsilon(I)$ of $I$ is defined as $$\varepsilon(I) = \limsup_{n\to \infty}\dfrac{\ell_R\left(H^0_{\bf m}(R/I^n)\right)}{n^d/d!},$$ where $\ell_R(-)$ denotes the length as an $R$-module.
In particular, if $I$ is ${\bf m}$-primary then it coincides with the usual multiplicity of $I$. Note that we can also write
$$\varepsilon(I) = \limsup_{n\to \infty}\frac{\ell_R\left(H^0_{\bf m}(R/I^n)\right)}{n^d/d!} =
\limsup_{n\to \infty}\frac{\ell_R\big({\widetilde{I^n}}/I^n\big)}{n^d/d!},$$
as by definition $H^0_{\bf m}(R/I^n) = {\widetilde{I^n}}/I^n$. It was shown by Cutkosky \cite{Cut14} that the `$\limsup$' in the definition can be replaced by `$\lim$' under mild conditions on the ring $R$.

\begin{thm}\label{epsilon}
Let $R$ and $I$ be as in Notations \ref{n1}. Then the following statements are true:
\begin{enumerate}
\item[$(i)$] The sequence $\{f_n(\varepsilon)\}_{n\in \N}$ converges (locally uniformly) to $f_{\varepsilon(I)}$ on the set $\R_{\geq 0}\setminus \{d_1\}$. Moreover, $$f_{\varepsilon(I)}(x) = f_{\{\widetilde {I^n}\}}(x)- f_{\{I^n\}}(x),\quad\mbox{for all}\;\; x\in \R_{\geq 0}.$$
 \item[$(ii)$] The function $f_{\varepsilon(I)}$ is continuous everywhere possibly except at $x=d_1$. Moreover, it is continuously differentiable outside the finite set $\{\alpha_I, d_1,\ldots, d_l\}$.
 \item[$(iii)$] The support of the function $f_{\varepsilon(I)}$ is contained in the closed interval $[\alpha_I, d(I)]$.
 \item[$(iv)$] The $\varepsilon$-multiplicity of $I$ is given by
 $$\varepsilon(I) =
 \int_0^{\infty} f_{\varepsilon(I)}(x)dx.$$
 \end{enumerate}
\end{thm}

We further obtained that all the above density functions remain invariant upto integral closure.

\begin{thm}\cite[Theorem 6.2]{DRT24}\label{6.1}
Let $R$ and $I$ be as in Notations \ref{n1}. Let $J=\overline{I}$ be the integral closure of the ideal $I$ in $R$. Then the following statements are true:
\begin{enumerate}
\item[$(i)$] $f_{\{I^n\}}(x) = f_{\{J^n\}}(x)$ for all $x \in \R_{\geq 0}$.
\item[$(ii)$] $f_{\{\widetilde {I^n}\}}(x) = f_{\{\widetilde {J^n}\}}(x)$ for all $x\in \R_{\geq 0}$.
\item[$(iii)$] $f_{\varepsilon(I)}(x) = f_{\varepsilon(J)}(x)$ for all $x \in \R_{\geq 0}$.
\end{enumerate}
\end{thm}

\subsection{Some basic facts about integral closures of ideals and integral extensions}

For an ideal $I$ in a commutative ring $R$, an element $r\in R$ is {\em integral} over $I$ it satisfies a monic polynomial
$$x^n + r_1x^{n-1} + r_2x^{n-2} + \cdots + r_n = 0,\quad\mbox{where}\;\; n\geq 1 \;\;\mbox{and}\;\; r_i\in I^i.$$
The set of integral elements over $I$ is an ideal, and is called the integral closure ${\overline I}$ of $I$.

If $A\to B$ is a homomorphism of commutative rings then $A\to B$ is an {\em integral extension} if every element $x$ of $B$ is integral over $A$, {\em i.e.}, it satisfies a monic polynomial
$$x^n + a_1x^{n-1} + a_2x^{n-2} + \cdots + a_n = 0,\quad\mbox{where}\;\; n\geq 1 \;\;\mbox{and}\;\; a_i\in A.$$

Integral closure behaves well with respect to localizations and integral extensions: 

\begin{enumerate}
 \item (Localization) If $I$ is an ideal in a commutative ring $R$ and $S$ is a multiplicatively closed set in $R$ then $S^{-1}(\overline{I}) = {\overline{S^{-1}I}}$.
\item (Integral extension)   If $A\subseteq B$ is an integral extension of commutative rings, and $I$ is an ideal in $R$. Then
$$\overline{IB}\cap A = \overline{I},$$
see \cite[Proposition 1.6.1]{HS06}.
\end{enumerate}

Now if $X$ is a scheme with structure sheaf $\sO_X$ and a sheaf of ideals $\mathcal{I}$ in $\mathcal{O}_X$, then the localization property of the integral closure of an ideal gives a well defined notion of the integral closure ${\overline{\sI}}$ of ${\sI}$ in $\sO_X$, which is given as follows. On an open affine subset $U$ of $X$ we define $H^0(U, \overline{\mathcal{I}}) = {\overline{H^0(U, \mathcal{I})}}$, where ${\overline {H^0(U, {\sI})}}$ is the integral closure of $H^0(U, \sI)$ in $H^0(U, \sO_X)$. In particular, if $X=\mathrm{Spec}\;R$ is affine then $H^0(X, \overline{\mathcal{I}}) = \overline{I}$ is the integral closure of $I$ in $R$.

Let $R=\oplus_{m\geq 0}R_m$ be a standard graded Noetherian algebra over a field $R_0$, $I$ be a homogeneous ideal of $R$ such that $R_1\not\subseteq I$, and $\mathcal{I}$ be the ideal sheaf associated to $I$ on $V = \mathrm{Proj}\;R$. Then the ideal sheaf associated to the integral closure ${\overline I}$ is the ideal sheaf ${\overline{\mathcal{I}}}$ in $\sO_V$.

We elaborate on this further. Let $x_1, \ldots, x_n$ be a set of degree $1$ generators of $R$. Then $V$ has open affine covering $\{D_+(x_i) =\mbox{Spec}\;R_{(x_i)}\}$, where $R_{(x_i)}$ is the subring of $R_{x_i}$ consisting of degree $0$ elements. Now by construction $H^0(D_+(x_i), \sO_X) = R_{(x_i)}$ and $H^0(D_+(x_i), \sI)$ is the set of degree $0$ elements of $I_{x_i}$. Therefore, if $\mbox{sheaf}(\overline{I})$ is the sheaf corresponding to the homogeneous ideal $\overline{I}$ then $H^0(D_+(x_i), \mbox{sheaf}(\overline{I}))$ is the integral closure $\overline{ H^0(D_+(x_i), \sI)}$ of $H^0(D_+(x_i), \sI)$ in $R_{(x_i)}$. So the notation $\mbox{sheaf}(\overline{I}) = \overline{\sI}$ is well-defined.

We recall the following facts from \cite[Definition 9.6.2 and Remark 9.6.4]{Laz04b}.

\begin{enumerate}
\item Let $X$ be a normal variety and $\mathcal{I}\subseteq \mathcal{O}_X$ be a nonzero ideal sheaf. Let $\nu\colon X^+\to X$ be the normalization of the blow up of $X$ along $\mathcal{I}$. Let $E$ be the exceptional divisor of $\nu$ so that $\mathcal{IO}_{X^+} = \mathcal{O}_{X^+}(-E)$. Then $\nu_*\mathcal{O}_{X^+}(-E) = \overline{\mathcal{I}}$.
\item Further, if  $f:Y\longrightarrow X$ is  a proper birational map surjective map between normal varieties with the property that ${\mathcal{IO}}_Y = \mathcal{O}_Y(-E)$ for some effective Cartier divisor $E$ on $Y$. Then $f$ factors through $\nu$ and consequently $f_*{\mathcal{O}}_Y(-E) = {\overline{\mathcal{I}}}$.
\end{enumerate}
 
\section{Finiteness of \texorpdfstring{$\ell_R({\overline J}/{\overline I})$}{l(J/I)} and the equality of their saturation density functions}

\begin{notations}\label{n2}
Adopt Notations \ref{nnotation}. Further assume that $R$ is a domain. Let ${\overline{I}}$ (resp. ${\overline{J}}$) denote the integral closure of $I$ (resp. $J$) in the ring $R$. Let $\mathcal{I}$ (resp. $\mathcal{J}$) be the ideal sheaf associated to $I$ (resp.  $J$) on $V=\mathrm{Proj}\;R$. Let $\overline{\mathcal{I}}$ (resp. $\overline{\mathcal{J}}$) denote the integral closure of the ideal sheaf $\mathcal{I}$ (resp. $\mathcal{J}$) in $\mathcal{O}_V$.
\end{notations}

In this section we show that $\ell_R(\overline{J}/{\overline{I}})<\infty$ is equivalent to saying that the saturation density functions of $I$ and $J$ are the same. In fact, we prove a stronger statement that if the saturation density functions of $I$ and $J$ agree on any single integer $c>{\bf d}$ (which we know is same as the assertion that their adic density functions agree on that point) then $\ell_R(\overline{J}/{\overline{I}})<\infty$.

For this assertion we provide two proofs:
\begin{enumerate}
 \item An algebraic proof using a result due to 
 Simis-Ulrich-Vasconcelos \cite[Theorem 3.3]{SUV01}. This proof requires no assumptions on the ground field $k$.
\item A geometric proof using a result due to 
Fulger-Koll\'ar-Lehmann \cite[Theorem $A$]{FKL16}. Since here in the geometric proof we use $\R$-divisors, the equality of the adic functions at any real number $x>{\bf d}$ would be sufficient. However, we need to further assume that the underlying field $k$ is perfect.
\end{enumerate}

\subsection{An algebraic proof}\label{algprf}
We fix an integer $c > {\bf d}$ and denote 
\[A:= \oOplus_{n\geq 0}A_n = \oOplus_{n \geq 0} {(I^n)}_{cn} \quad \mbox{ and } \quad  B:=   \oOplus_{n\geq 0}B_n = \oOplus_{n \geq 0} {(J^n)}_{cn}.\]

Since $R$ is a standard graded $k$-algebra it is easy to check that both $A$ and $B$ are standard graded $k$-algebras, see \cite[Lemma 2.2(a)]{HT03}. Further it is known that $\dim A = d = \dim B$, see \cite[Lemma 2.2(b)]{HT03}. However, this can also be checked by using the density functions as follows.
$$\lim_{n\to \infty}\frac{\ell_k(A_n)}{n^{d-1}/d!} = \lim_{n\to \infty}\frac{\ell_k\left((I^n)_{\lfloor cn\rfloor}\right)}{n^{d-1}/d!} = f_{\{I^n\}}(c)\quad\mbox{and}\quad
\lim_{n\to \infty}\frac{\ell_k(B_n)}{n^{d-1}/d!} =\lim_{n\to \infty}\frac{\ell_k\left((J^n)_{\lfloor cn\rfloor}\right)}{n^{d-1}/d!} = f_{\{J^n\}}(c),$$
where both $f_{\{I^n\}}(c)$ and $f_{\{I^n\}}(c)$ are nonzero real numbers by Theorem \ref{2.3}. 

The next lemma follows from \cite[Theorem 3.3]{SUV01}, which was proved in a more general setup. Here we give a self-contained proof in our setup following the same line of arguments as in \cite{SUV01}.
 
\begin{lemma}\label{diag_intext}
Let $A$, $B$ and $c > {\bf d}$ be as in Subsection \ref{algprf}. If  $f_{\{I^n\}}(c) = f_{\{J^n\}}(c)$ then the canonical inclusion map $A\hookrightarrow B$ is an integral extension of rings.
\end{lemma}

\begin{proof} Since $A_1B$ is a graded ideal in $B$ we can consider the Rees Algebra 
$\sR(A_1B) = B[(A_1B)T]$ and the associated graded ring $G= \mathrm{gr}_{(A_1B)}(B)$ of the ideal $A_1B$ in $B$.

Note that $G$ is a standard graded (with a suitable  grading) equidimensional ring of dimension $d$. This we can argue as follows.
The extended Rees algebra ${\sR} = \sR(A_1B)[T^{-1}]$ is an affine $k$-algebra which is an integral domain and $T^{-1}$ is a homogeneous nonzerodivisor. Since $G= \sR/T^{-1}\sR$, the catenary property of this ring implies that $G$ is an equidimensional ring of dimension $d$.

If we endow bigrading to $\sR$ by assigning bidegree $(1,0)$ to the elements of $B_1$ and bidegree $(0,1)$ to the elements of $A_1$, then it gives an $\N$-grading on $G$, namely the total grading. Now it is obvious that with respect to this grading $G$ is a standard graded $k$-algebra and is generated by degree one elements $B_1$ and $A_1B_0$.
 
In particular, the components of the finitely generated homogeneous ideal $B_1G$ has polynomial growth: $B_1G$ is a finitely generated graded $G/(0:_GB_1G)$-module. Let ${\tilde d}$ be the Krull dimension of $G/(0:_G B_1G)$. Then by the Hilbert-Serre theorem there is a polynomial $q(x)\in \Q[x]$ of degree ${\tilde d}-1\leq d-1$ such that $\ell_k(B_1G)_n = q(n)$ for all $n\gg 0$.

\vspace{5pt}

\noindent{\bf Claim}. The coefficient of $n^{d-1}$ in $q(x)$ is zero, equivalently ${\tilde d}<d$.

\vspace{5pt}

\noindent{Proof of the claim}.
By construction $$\ell_k\left((B_1G)_n\right) = \ell_k(B_n/A_n) = \ell_k((J^n)_{cn}/(I^n)_{cn}),$$ whereas by Theorem \ref{2.3} $(i)$,
$$f_{\{J^n\}}(c) = \lim_{n \to \infty}\frac{\ell_k \left((J^n)_{cn}\right)}{n^{d-1}/d!}\quad\mbox{and}\quad f_{\{I^n\}}(c) = \lim_{n \to \infty}\frac{\ell_k \left((I^n)_{cn}\right)}{n^{d-1}/d!}.$$
Therefore, if $e_1(A,B)$ denote $(d-1)!$ times the leading coefficient of $q(x)$ then
$$e_1(A,B) = \lim_{n \to \infty}\frac{\ell_k \left(B_n/A_n\right)}{n^{d-1}/(d-1)!} = \frac{1}{d}\left[f_{\{J^n\}}(c)-f_{\{I^n\}}(c)\right]=0.$$
Now ${\tilde d}-1< d-1$ implies that the ideal $(0:_GB_1G)$ has positive height. Since $G$ is equidimensional it implies that $B_1G$ is contained inside the nilradical of $G$ and therefore we have $B_1\subseteq \sqrt{A_1B}$. 
Then there is an $m$ such that $B_1^m\subseteq A_1B$.
Since $B$ is generated by degree $1$ elements, we have
$B_1^m = A_1B_{m-1} = A_1B_1^{m-1}$. In particular, $B$ is integral over $A$.
\end{proof}

\begin{propose}\label{sing_pt_adic}
Let $A$, $B$ and $c > {\bf d}$ be as in Subsection \ref{algprf}. If $f_{\{I^n\}}(c) = f_{\{J^n\}}(c)$ then $\ell_R(\overline{J}/\overline{I})< \infty$.
\end{propose}
\begin{proof} 
First we show that $(\overline{J})_c=(\overline{I})_c$. Let $r \in (\overline{J})_c$. Then there exists an integer $t>0$ and elements $b_1, \ldots, b_t$ in $R$ such that $b_i \in J^i$ for all $i=1,\ldots,t$ and
\begin{equation}\label{integral_eq}
r^t+b_1r^{t-1}+\cdots+b_{t-1}r+b_t=0.
\end{equation}
We may assume that all the $b_i$'s are homogeneous. Then $\deg b_i=tc-(t-i)c=ci$ for $i=1,\ldots, t$ and hence $b_i \in {(J^i)}_{ci}=B_i \subseteq B$. Hence $r$ is integral over $B$ and therefore by Lemma \ref{diag_intext}, $r$ is integral over the ring $A$. In particular, there exists an integer $s>0$ and (homogeneous) elements $a_1, \ldots, a_s$ in $A$ such that
\[r^s+a_1r^{s-1}+\cdots+a_{s-1}r+a_s=0.\]
Clearly, $\deg a_i=sc-(s-i)c=ci$. So $a_i \in (I^i)_{ci} = A_i \subset A$. Consequently, $r \in \overline{I} \cap R_c=(\overline{I})_c$.

By \eqref{integral_eq} we know that $(\overline{J})_m \subseteq \overline{\oplus_{i\geq m}J_i}$. Therefore, to show $\ell_R(\overline{J}/\overline{I})<\infty$, it is enough to show that $J_m \subseteq {\overline I}$ for all integers $m\geq c$.

Let $y_1, \ldots, y_s$ be a set of homogeneous generators of $J$ such that $e_j=\deg y_j \leq c$. Since $R$ is standard graded we have $y_jR_{c-e_j} \subseteq (J)_c\subseteq (\overline{J})_c = (\overline{I})_c$. Therefore, $y_jR_m \subseteq {\overline{I}}$ for all $m\geq c-e_j$.

Now let  $y\in J_m$ where $m\geq c$. Then $y$ is a sum of elements of the type $ry_1^{i_1}\cdots y_s^{i_s}$, where $r\in R_{m-\sum_ii_je_j}$ and one of the exponents, say, $i_1\geq 1$. Then $ry_1^{i_1}\cdots y_s^{i_s} = y_1(ry_1^{i_1-1}\cdots y_s^{i_s})\in y_1R_{m-e_1}\subseteq {\overline{I}}$. This proves that $J_m\subseteq {\overline{I}}$ for all $m\geq c$.
\end{proof}

\subsection{A geometric proof}

Before we start with a geometric proof we recall  the following geometric interpretation of the finiteness of $\ell_R({\overline J}/{\overline I})$.

\begin{lemma}\label{saturation}
Let $R$, $I$ and $J$ be as in Notations \ref{n2}. Then the following statements are equivalent:
\begin{enumerate}
 \item[$(i)$]  $\ell_R({\overline{J}}/{\overline{I}}) <\infty$.
 \item[$(ii)$] $\overline{\mathcal{I}} = \overline{\mathcal{J}}$.
 \end{enumerate}
\end{lemma}

\begin{proof}
\noindent{$(ii)\implies (i)$}. Since $R$ is a domain, $H^0_{\bf m}(R) = 0$ and there is an integer $m_1$ such that $(H^1_{\bf m}(R))_m=0$ for all integers $m\geq m_1$. Therefore for all $m\geq m_1$, we have
$$H^0(V, \overline{\mathcal{I}}(m)) = (\overline{I}\colon_R \mathbf{m}^{\infty})_m \quad\mbox{and}\quad H^0(V, \overline{\mathcal{J}}(m)) = (\overline{J}\colon_R \mathbf{m}^{\infty})_m.$$
Now $(i)$ follows because 
$$\ell_R\left(\frac{\overline{J}}{\overline{I}}\right) \leq \ell_R\left(\frac{\overline{J}\colon_R\mathbf{m}^{\infty}}{{\overline{I}}}\right)
= \ell_R\left(\frac{\overline{J}\colon_R \mathbf{m}^{\infty}}{\overline{I}\colon_R \mathbf{m}^{\infty}}\right) +\ell_R\left(\frac{\overline{I}\colon_R \mathbf{m}^{\infty}}{\overline{I}}\right) = 
\ell_R\left(\frac{\overline{I}\colon_R \mathbf{m}^{\infty}}{\overline{I}}\right) <\infty.$$
  
\vspace{5pt}
\noindent{$(i)\implies (ii)$}. We have $\overline{I}\colon_R {\bf m}^{\infty} = \overline{J}\colon_R {\bf m}^{\infty}$. It is enough to show that $\overline{\sJ} \subseteq \overline{\sI}$. Let $h_1, \ldots, h_s$ be a set of homogeneous generators of $\overline{J}$, where $\deg h_j = e_j$. Let $x_1,\ldots x_n$ be a set of degree one generators of $R$. Then $\cup_{i=1}^n D_+(x_i)$ is an open affine covering of $V=\mathrm{Proj}\;R$ and $H^0(D_+(x_i), \overline{\sJ})$ is generated as an $H^0(D_+(x_i), \sO_V)$-module by the elements $h_1/x_i^{e_1},\ldots,h_s/x_i^{e_s}$. But for a given $j$ there is an $m_j$ such that $h_jx_i^{m_j}\in \overline{I}$, which gives
$h_j/x_i^{e_j} = (h_jx_i^{m_j})/(x_i^{m_j+e_j}) \in H^0(D_+(x_i), \overline{\sI})$.
Since $\overline{\sI}$ and $\overline{\sJ}$ are determined by the sections on any given open affine cover of $V$, we have $\overline {\sJ}\subseteq \overline{\sI}$.
\end{proof}

\vspace{5pt}

Now we recall the following result of Fulger-Koll\'ar-Lehmann.

\begin{thm}\cite[Theorem A]{FKL16}
Let $X$ be a proper, normal algebraic variety over a perfect field $k$. Let $D$ be a big nef Cartier $\R$-divisor and $E$ be an effective Cartier $\R$-divisor on $X$. Then
$$\mathrm{vol}_X(D-E) = \mathrm{vol}_X(D) \iff \mbox{the divisor}\; E=0.$$
\end{thm}

\begin{propose}\label{impsat1}
Adopt Notations \ref{n2} and further assume that the field $k$ is perfect. If $f_{\{I^n\}}(x) = f_{\{{J^n}\}}(x)$ for some real number $x>\bf d$ then $\overline{\mathcal{I}} = \overline{\mathcal{J}}$.
\end{propose}

\begin{proof}
Let $x\in \R$ such that $x>\mathbf{d}$ and $f_{\{\widetilde{I^n}\}}(x) = f_{\{\widetilde{J^n}\}}(x)$.
Let $$\varphi_I \colon Z_I = \mathrm{Bl}_{\mathcal{I}}(V) \longrightarrow V \quad \text{and}\quad \varphi_J \colon Z_J = \mathrm{Bl}_{\mathcal{J}}(V) \longrightarrow V $$ denote the blow-ups of $V$ along $\mathcal{I}$ and $\mathcal{J}$ with the exceptional divisors $F_I$ and $F_J$ respectively. Let $L_I$ and $L_J$ be the pullbacks of a hyperplane section on ${V}$ to $Z_I$ and $Z_J$ respectively. From Theorem \ref{mainsat}, we have
$$f_{\{\widetilde{I^n}\}}(x) = d\cdot\mathrm{vol}_{Z_I}(xL_I-F_I) \quad \text{and}\quad f_{\{\widetilde{J^n}\}}(x) = d\cdot\mathrm{vol}_{Z_J}(xL_J-F_J),$$
where the Cartier $\R$-divisors $xL_I-F_I$ and $xL_J-F_J$ are both big and nef, since $x>\bf d$.

Let $\varphi_{IJ}\colon Z_{IJ} =\mathrm{Bl}_{\mathcal{IJ}}({ V}) \longrightarrow { V}$ be the blow-up of ${ V}$ along $\mathcal{IJ}$. Then it is also the blow-up of $Z_J$ along $\mathcal{IO}_{Z_J}$ and the blow-up of $Z_I$ along $\mathcal{JO}_{Z_I}$, see \cite[\href{https://stacks.math.columbia.edu/tag/080A}{Lemma 010A}]{Sta24}. (In fact, it can be checked that the induced maps $\psi_J:Z_{IJ}\longrightarrow Z_J$ and $\psi_I:Z_{IJ}\longrightarrow Z_I$ are both finite morphisms. However, we do not need this observation here).

Let $\eta:{\overline V}\longrightarrow V$ be the normalization of $V$. Let $X^+$ denote  the normalization of the blow-up $\mathrm{Bl}_{\mathcal{IJO}_{\overline  V}}({ \overline V})\longrightarrow {\overline V}$ of ${\overline V}$ along the ideal sheaf $\mathcal{IJO}_{\overline V}$, and let  the canonical induced  map be denoted as $\nu\colon X^+  \longrightarrow {\overline V}.$
Therefore we have following commutative diagrams
{\small \[\xymatrix{& Z_{IJ} \ar[ld]_{\psi_I} \ar[rd]^{\psi_J}&\\
Z_I \ar[rd]_{\varphi_I} & & Z_J \ar[ld]^{\varphi_J}\\
& V &}\qquad \qquad \xymatrix{&X^+\ar[d]\ar[rdd]^\nu\ar[ldd]_{\overline{\eta}}&\\& \mathrm{Bl}_{\mathcal{IJO}_{\overline  V}}({ \overline V}) \ar[ld] \ar[rd]&\\
	Z_{IJ} \ar[r]_{\varphi_{IJ}} &  V & \overline{V} \ar[l]^{\eta}},\]}
where
the map $\eta\circ\nu:X^+\longrightarrow V$ factors through the canonical map ${\overline \eta}\colon X^+\longrightarrow Z_{IJ}$. Note that both $\mathcal{I}\mathcal{O}_{Z_{IJ}}$ and $\mathcal{J}\mathcal{O}_{Z_{IJ}}$ are locally principal on $Z_{IJ}$ and $\mathcal{I}\mathcal{O}_{Z_{IJ}} \subseteq \mathcal{J}\mathcal{O}_{Z_{IJ}}$. Therefore, $Z_{IJ}$ carries effective Cartier divisors $E_0$ and $E_J$ with $E_J = \psi_J^*F_J$ and $E_0+E_J = \psi_I^*F_I$
so  that 
$$\mathcal{J}\mathcal{O}_{Z_{IJ}} = \mathcal{O}_{Z_{IJ}}(-E_J) \quad\text{and}\quad \mathcal{I}\mathcal{O}_{Z_{IJ}} = \mathcal{O}_{Z_{IJ}}(-E_J-E_0).$$ 

Let $H_1 = \psi_I^*L_I = \psi_J^*L_J$ and  
$H = {\overline \eta}^*\psi_I^*L_I = {\overline \eta}^*\psi_J^*L_J = {\overline \eta}^*H_1$. Now the birational invariance of volume implies that
 \begin{align*}
   &\mathrm{vol}_{Z_I}(xL_I-F_I) = \mathrm{vol}_{Z_{IJ}}(xH_1 - E_J - E_0) = \mathrm{vol}_{X^+}(xH -{\overline \eta}^*E_J - {\overline \eta}^*E_0),\\
  &\mathrm{vol}_{Z_J}(xL_J-F_J) = \mathrm{vol}_{Z_{IJ}}(xH_1 - E_J) = \mathrm{vol}_{X^+}(xH - {\overline \eta}^*E_J).
 \end{align*}
The $\mathbb{R}$-Cartier divisor $xH- {\overline \eta}^*E_J = {\overline {\eta}}^*\psi_J^*(xL_J-F_J)$ is both big and nef on  the normal variety $X^+$ and
$$\mathrm{vol}_{X^+}(xH - {\overline \eta}^*E_J -{\overline \eta}^*E_0) = \mathrm{vol}_{X^+}(xH -  {\overline \eta}^*E_J).$$
Therefore, by \cite[Theorem A]{FKL16} we conclude that ${\overline \eta}^*E_0=0$. In other words, 
$$\mathcal{I}\mathcal{O}_{X^+} = 
\mathcal{O}_{X^+}(-{\overline \eta}^*E_I)
= \mathcal{O}_{X^+}(-{\overline \eta}^*E_J-{\overline \eta}^*E_0) = 
\mathcal{O}_{X^+}(-{\overline \eta}^*E_J) = \mathcal{J}\mathcal{O}_{X^+}.$$ Hence
 $$\overline{\mathcal{JO}_{\overline V}} = \nu_*\mathcal{J}\mathcal{O}_{X^+} = \nu_*\mathcal{I}\mathcal{O}_{X^+} = \overline{\mathcal{IO}_{\overline V}}$$
 as ${\overline V}$ is normal. Now since $\eta\colon \overline{V}\longrightarrow V$ is a finite map, and the fact that integral closure extends and contracts from integral extensions (\cite[Proposition 1.6.1]{HS06}) we have $\overline{\mathcal{J}}  = \overline{\mathcal{I}}$.
 \end{proof}

\begin{thm}\label{impsat}
Let $R$, $I$ and $J$ be as in Notations \ref{n2}. Then the following statements are equivalent:
\begin{enumerate}
 \item[$(i)$] $f_{\{\widetilde{I^n}\}}(x) = f_{\{\widetilde{J^n}\}}(x)$ for all $x\in\mathbb{R}_{\geq 0}$.
 \item[$(ii)$] $f_{\{\widetilde{I^n}\}}(x) = f_{\{\widetilde{J^n}\}}(x)$ for some integer $x>\bf d$. Further, we can choose $x$ to be real
  number if $k$ is perfect.  
 \item[$(iii)$] $\ell_R({\overline{J}}/{\overline{I}}) <\infty$.
 \end{enumerate}
\end{thm}
\begin{proof} $(i)\implies (ii)$  is obvious. 

\vspace{5pt}
\noindent{$(ii)\implies (iii)$} Note that on the interval $({\bf d}, \infty)$ the saturation density functions coincide with their respective adic density functions. Therefore, the first part of the assertion $(ii)$ implies $(iii)$ follows from Proposition \ref{sing_pt_adic}, and the second part of the same assertion follows from Proposition \ref{impsat1}.
 
\vspace{5pt}
\noindent{$(iii)\implies (i)$}. The hypothesis $\ell_R({\overline{J}}/{\overline{I}})<\infty$ implies that ${\overline{J}}^n:_R{\bf m}^{\infty} = {\overline{I}}^n:_R{\bf m}^{\infty}$. Therefore the saturation density functions for ${\overline I}$ and ${\overline J}$ are same.
On the other hand, the saturation density function of an ideal is same as the saturation density function of its integral closure, see part $(ii)$ of Theorem \ref{6.1}.
 \end{proof}

\begin{thm}\label{adicfunction}
Let $R$, $I$ and $J$ be as in Notations \ref{n2}.  Then the following statements are equivalent:
\begin{enumerate}
 \item[$(i)$] $f_{\{I^n\}}(x) = f_{\{J^n\}}(x)$ for all $x\in\mathbb{R}_{\geq 0}$.
  \item[$(ii)$] $f_{\{I^n\}}(x) = f_{\{J^n\}}(x)$ for almost all $x\in\mathbb{R}_{\geq 0}$ except at finitely many points.
 \item[$(iii)$] $f_{\{\widetilde{\mathsf{I}^n}\}}(x) =
 f_{\{\widetilde{\mathsf{J}^n}\}}(x)$ for all $x\in\mathbb{R}_{\geq 0}$.
 
 \item[$(iv)$] $\overline{I} = \overline{J}$.
 \end{enumerate}
\end{thm}
\begin{proof}
$(i) \implies (ii)$ is obvious.

\vspace{5pt}
\noindent{$(ii) \implies (iii)$}.
Recall from Notations \ref{nnotation} that $S=R[y]$,
$\mathbf{n}=\mathbf{m}+(y)$, $\mathsf{I}=IS$, and $\mathsf{J}=JS$. Note that if $I^n
= \cap_{i=1}^t Q_{i,n}$ is a primary decomposition in $R$, then
$\mathsf{I}^n=\cap_{i=1}^t Q_{i,n}S$ is a primary decomposition in $S$, see
\cite[Chapter 4, Exercise 7]{atiyah}. Thus, for all integers $n>0$, one has
$$\widetilde{\mathsf{I}^n}:= \mathsf{I}^n\colon_S \mathbf{n}^{\infty} =
\mathsf{I}^n, \quad \text{and} \quad \widetilde{\mathsf{J}^n}:= \mathsf{J}^n\colon_S
\mathbf{n}^{\infty} = \mathsf{J}^n.$$ Further observe that ${(\mathsf{I}^n)}_m =
\sum_{i=1}^m {(I^n)}_i\cdot y^{m-i}$. So 
\begin{equation}\label{satdensity}f_{\{\mathsf{I}^n\}}(x) :=
\lim_{n\to\infty}\dfrac{\ell_k\left((\mathsf{I}^n)_{\lfloor xn
\rfloor}\right)}{n^d/(d+1)!} = \lim_{n\to\infty}\dfrac{\sum_{m=1}^{\lfloor xn
\rfloor}\ell_k\left((I^n)_m\right)}{n^d/(d+1)!} = (d+1)\cdot\int_0^x
f_{\{I^n\}}(y)dy,\end{equation}
where the last equality is a consequence of Lebesgue's dominated convergence theorem (see  part $(iii)$ of Theorem
\ref{2.3}). For all $x\in \mathbb{R}_{\geq 0}$, we have
$$f_{\{\widetilde{\mathsf{I}^n}\}}(x) = f_{\{\mathsf{I}^n\}}(x) = (d+1)\cdot\int_0^x
f_{\{I^n\}}(y)dy, \quad \text{and} \quad f_{\{\widetilde{\mathsf{J}^n}\}}(x) =
f_{\{\mathsf{J}^n\}}(x) = (d+1)\cdot\int_0^x f_{\{J^n\}}(y)dy.$$ 

\vspace{5pt}

\noindent{$(iii) \implies (iv)$}.
Now applying Theorem \ref{impsat} with $R$, $I$, and $J$ replaced with $S$, $\mathsf{I}$, and $\mathsf{J}$ respectively, gives that
$\overline{\mathsf{I}}\colon_S \mathbf{n}^{\infty} = \overline{\mathsf{J}}\colon_S
\mathbf{n}^{\infty}$. By imitating an argument from \cite[Proposition 5.2.1]{HS06}
one can verify that $\overline{\mathsf{I}}:= \overline{IR[y]} = \sum_{i\in
\mathbb{N}}\overline{I}\cdot y^i = \overline{I}S$, and so
$\overline{\mathsf{I}}\colon_S \mathbf{n}^{\infty} = \overline{I}S$. In other words,
$\overline{I}S = \overline{J}S$ and intersecting this with $R$ gives us $(iv)$.

\vspace{5pt}

\noindent{$(iv) \implies (i)$} follows from Theorem \ref{6.1}.
\end{proof}

\section{Criteria for \texorpdfstring{$\ell_R({\overline{J}}/{\overline{I}})<\infty$}{l(J/I) finite} and for \texorpdfstring{${\overline{I}} ={\overline{J}}$}{I=J} in terms of multiplicities}

\begin{notations}\label{n3}
Adopt Notations \ref{nnotation}. Further assume that $R$ is a domain. From \cite{Rob98} or \cite[Theorem 4.2]{HT03}, we know that there exist constants $m_0\geq 0$ and $n_0\geq 0$ such that for all integers $m\geq d(I)n+m_0$ and $n \geq n_0$, the length function $\ell_k \left((I^n)_m\right)$ agrees with a polynomial in $m$ and $n$, i.e.,
$$\ell_k \left((I^n)_m\right) = \sum_{i=0}^{d-1}\frac{e_i(R[It])}{i!(d-1-i)!}m^{i}n^{d-1-i}+ \mbox{lower degree terms},$$ where the coefficients $e_i(R[It])$ are integers for all $i=0,\ldots,d-1$. We shall refer to $e_i(R[It])$ as the \emph{$i^{\mathrm{th}}$ RA-multiplicity of $R[It]$}, where `RA' stands for the Rees algebra.
\end{notations}

\begin{rmk}\label{firstremark}
The $i^{\mathrm{th}}$ RA-multiplicity $e_i(R[It])$ is the intersection number $H^i\cdot E^{d-1-i}$, where we recall that $E$ is the exceptional divisor  of the blow up $\pi\colon X\to V=\mbox{Proj}\;R$ along the ideal sheaf $\mathcal{I}$ (corresponding to $I$), and $H$ is the pullback of a hyperplane section on $V$. By Theorem~\ref{mainsat}, for all real numbers $x\geq d(I)$, the saturation density function for $I$ is given by
 \begin{equation}\label{mm}
f_{\{\widetilde{I^n}\}}(x) = d\cdot\sum_{i=0}^{d-1}\binom{d-1}{i}e_i(R[It])x^{i} =d\cdot\sum_{i=0}^{d-1}(-1)^{d-1-i}\binom{d-1}{i}\left(H^{i}\cdot E^{d-1-i}\right)x^{i}.
\end{equation}

Further, $e_{d-1}(R[It]) = H^{d-1} = e(R)$ and therefore is independent of the ideal $I$. We shall later see in Lemma \ref{lem1} that $e_i(R[It]) = 0$ for all $i$ satisfying $\dim(R/I)\leq i <d-1$.
\end{rmk}

Now we are ready to give a numerical characterization of the properties $(1)$ $\ell_R(\overline{J}/\overline{I}) <\infty$, and $(2)$ $\overline{I} = \overline{J}$ in terms of various multiplicities.
 
\begin{thm}\label{integral1}
Let $R$, $I$ and $J$ be as in Notations \ref{n3}. Then the following statements $(i)$, $(ii)$, and $(iii)$  are equivalent:
\begin{enumerate}
 \item[$(i)$] $\ell_R ({\overline{J}}/\overline{I}) <\infty$.
 \item[$(ii)$] $e_i(R[It]) = e_i(R[Jt])$ for all $i$ where $0\leq i <\dim(R/I)$.
 \item[$(iii)$] $e\big({R[It]}_{\Delta_{(c,1)}}\big) =  e\big({R[Jt]}_{\Delta_{(c,1)}}\big)$ for some (every) integer  $c> {\bf d}$.
\end{enumerate}
 The following statements $(a)$, $(b)$, and $(c)$ are also equivalent:
\begin{enumerate}
 \item[$(a)$] $\overline{I} = \overline{J}$.
 \item[$(b)$] $\varepsilon(I) = \varepsilon(J)$ and $e_i(R[It]) = e_i(R[Jt])$ for all $0\leq i <  \dim(R/I)$.
 \item[$(c)$] $\varepsilon(I) = \varepsilon(J)$ and $e\big({R[It]}_{\Delta_{(c, 1)}}\big) =  e\big({R[Jt]}_{\Delta_{(c, 1)}}\big)$ for some (every) integer $c> {\bf d}$.
 \end{enumerate}
\end{thm}

\begin{proof}
We recall that $f_{\{I^n\}}(x) = f_{\{\widetilde{I^n}\}}(x) $ for all $x\geq  d(I)$ and $f_{\{J^n\}}(x) =f_{\{\widetilde{J^n}\}}(x)$ for all $x\geq d(J)$.

\vspace{5pt}

\noindent{$(i) \iff (iii)$.} For any integer $c>{\bf d}$, we note that both the diagonal subalgebras ${R[It]}_{\Delta_{(c, 1)}}$ and ${R[It]}_{\Delta_{(c, 1)}}$ are $d$-dimensional standard graded domains over $k$, see \cite[Lemma 2.2]{HT03}. So we have a well-defined notion of their Hilbert-Samuel multiplicities, i.e.,
$$e\big({R[It]}_{\Delta_{(c, 1)}}\big) = \lim_{n\to \infty} \frac{{\ell_k\left((I^{n})_{cn}\right)}}{{n^{d-1}/(d-1)!}} \quad\mbox{and}\quad e\big({R[Jt]}_{\Delta_{(c, 1)}}\big) = \lim_{n\to \infty} \frac{{\ell_k\left((J^{n})_{cn}\right)}}{{n^{d-1}/(d-1)!}}.$$

On the other hand
from Definition~\ref{d1} and Theorem~\ref{2.3} we have
\begin{equation*}
f_{\{I^n\}}(c) = d\cdot e\big({R[It]}_{\Delta_{(c,1)}}\big) \quad\mbox{and}\quad f_{\{J^n\}}(c) = d\cdot e\big({R[Jt]}_{\Delta_{(c,1)}}\big).
\end{equation*}
In particular, $e\big({R[It]}_{\Delta_{(c,1)}}\big) =  e\big({R[Jt]}_{\Delta_{(c,1)}}\big)$ if and only if $f_{\{\widetilde{I^n}\}}(c) = f_{\{\widetilde{J^n}\}}(c)$. Now the assertion follows from Theorem \ref{impsat}.

\vspace{5pt}
 
\noindent{$(i) \iff (ii)$.} This follows from Remark \ref{firstremark} and Theorem \ref{impsat}.

\vspace{5pt}
 
\noindent{$(a) \implies (b)$.} We only need to show that 
$\varepsilon(I) = \varepsilon(J)$, which follows from 
 Theorem \ref{adicfunction} and Theroem \ref{epsilon} parts $(i)$ and $(iv)$.

\vspace{5pt}
 
\noindent{$(b) \implies (a)$.} From \eqref{mm} we have $f_{\{\widetilde{I^n}\}}(x) = f_{\{\widetilde{J^n}\}}(x)$ for all $x\geq {\bf d}$ and thus by Theorem \ref{impsat} this equality holds for all $x\geq 0$. From Theorem \ref{epsilon} we also know that
$$f_{\varepsilon(I)}(x) = f_{\{\widetilde{I^n}\}}(x) - f_{\{I^n\}}(x),\quad f_{\varepsilon(J)}(x) = f_{\{\widetilde{J^n}\}}(x) - f_{\{J^n\}}(x)$$ and $f_{\varepsilon(I)}(x) = 0 = f_{\varepsilon(J)}(x)$ for all $x > {\bf d}$. This gives that
\begin{equation}\label{diff_epsilon}
\varepsilon(I)-\varepsilon(J) = \int_0^{\bf d} \left(f_{\varepsilon(I)}(x)-f_{\varepsilon(J)}(x)\right)dx = \int_0^{\bf d} \left(f_{\{J^n\}}(x)-f_{\{I^n\}}(x)\right)dx.
\end{equation}
But $f_{\{J^n\}}- f_{\{I^n\}}$ is a nonnegative function which is continuous everywhere except possibly at (at most) two points. Therefore, $\varepsilon(I) = \varepsilon(J)$ implies that $f_{\{J^n\}} =  f_{\{I^n\}}$ almost everywhere on the closed interval $[0, {\bf d}]$ and hence almost everywhere on $[0, \infty)$. Thus ${\overline{I}} = {\overline{J}}$ by Theorem \ref{adicfunction}.

\vspace{5pt}
 
\noindent{$(b) \iff (c)$.} This follows from the equivalence $(ii)\iff (iii)$.
\end{proof}

\begin{cor}\label{equigen}
Let $R$, $I$ and $J$ be as in Notations \ref{n3}. Further assume that both $I$ and $J$ are generated in equal degrees $d(I)$ and $d(J)$ respectively. Then following statements are equivalent:
\begin{enumerate}
 \item[$(i)$] $\overline{I} = \overline{J}$.
 \item[$(ii)$] $d(I)=d(J)$ and $e_i(R[It]) = e_i(R[Jt])$ for all $0\leq i < \dim(R/I)$.
 \item[$(iii)$] $d(I)=d(J)$ and $e\big({R[It]}_{\Delta_{(c,1)}}\big) =  e\big({R[Jt]}_{\Delta_{(c,1)}}\big)$ for some (every) integer  $c> {\bf d}$.
\end{enumerate}
\end{cor}
\begin{proof}
We only need to prove $(i)\iff (ii)$.

\vspace{5pt} 

Note  that
$d(I)\geq d(J)$. We have
$f_{\{I^n\}}(x) = 0$ if $x< d(I)$, $f_{\{J^n\}}(x) = 0$ if $x<d(J)$, and $f_{\{J^n\}}(x) = f_{\{\widetilde{J^n}\}}(x)>0$ if $x>d(J)$.

\vspace{5pt}

\noindent{$(i) \implies (ii)$.} Second part of $(ii)$ follows from Theorem~\ref{epsilon}. 
 If $d(I)> d(J)$ then $f_{\{J^n\}}(x) - f_{\{I^n\}}(x) = f_{\{J^n\}}(x) > 0$ for all $x\in (d(J), d(I))$, which contradicts the assertion
of Theorem \ref{adicfunction}. 

\vspace{5pt}

\noindent{$(ii) \implies (i)$.} It is enough to show $\varepsilon(I) = \varepsilon(J)$. Theorem \ref{impsat} and the hypothesis on the RA-multiplicities imply that $f_{\{\widetilde{I^n}\}}(x) = f_{\{\widetilde{J^n}\}}(x)$ for all $x\geq 0$. Therefore, if $ d(I) = d(J)$ then
{\small 
$$\varepsilon(J)-\varepsilon(I) = 
\int_0^{d(J)} 
\big(f_{\{\widetilde{J^n}\}}- f_{\{J^n\}}\big)(x)dx - \int_0^{d(J)}\big(f_{\{\widetilde{I^n}\}} - f_{\{I^n\}}\big)(x)dx =     \int_0^{d(J)}
\big(f_{\{I^n\}}- f_{\{J^n\}}\big)(x)dx =   0.$$}
 \end{proof}

\section{Multiplicities and integral dependence}\label{mixmul}

Let $(R,{\bf m})$ be a Noetherian local ring of dimension $d$. Let $I$ be an ${\bf m}$-primary ideal and $J$ be an arbitrary ideal with positive height. Bhattacharya \cite{Bha57} showed that for all $u\gg 0$ and $v\gg 0$, the bivariate Hilbert function $\ell_R \left(I^{u}J^v/I^{u+1}J^v\right)$ agrees with a bivariate numerical polynomial $Q(u,v)$ of total degree $d-1$. Moreover, we can write
\begin{equation}\label{mixedMul}
Q(u,v) = \sum_{i=0}^{d-1} \dfrac{e_i\left(I \vert J\right)}{i!(d -1-i)!}u^{d-1-i}v^i + \text{lower degree terms},
\end{equation}
where the coefficients $e_i\left(I\vert J\right)$ are integers. Risler and Tessier \cite{Tei73} studied these coefficients when $J$ is ${\bf m}$-primary and called them \emph{mixed multiplicities of $I$ and $J$}. Later in \cite{KV89}, Katz and Verma analyzed the coefficients when $J$ is an arbitrary ideal, also see \cite{Tru01}.

\begin{notations}\label{n4}
Let $R$, $\bf m$, and $k$ be as in Notations \ref{nnotation}. Further assume that $R$ is a domain. Let $I$ be a nonzero homogeneous ideal in $R$. Let $\beta>0$ be an integer such that $I$ is generated in degrees $\leq \beta$ and $I_{\geq\beta}$ denotes the truncated ideal $\oplus_{m\geq \beta}I_m$.
\end{notations}

The RA-multiplicities (see Notations \ref{n3}) and the mixed multiplicities are related to each other as follows. From \cite[Lemma 5.3]{DDRV24} and the proof of \cite[Proposition 5.4]{DDRV24}, one has
\begin{align}
e_{i}\big(R[It]\big)&=\sum_{j=0}^{d-1-i}(-1)^j \binom{d-1-i}{j}{\beta}^{j}e_{d-1-i-j}\big(\mathbf{m}|I_{\geq\beta}\big) &&\mbox{for all } i=0,\ldots,d-1, \;\text{and}
\label{imprel1}\\
e_{i}\big(\mathbf{m}|I_{\geq\beta}\big) &= \sum_{j=0}^i \binom{i}{j}{\beta}^{j}e_{d-1-i+j}(R[It]) &&\mbox{for all } i=0,\ldots,d-1.\label{imprel2}
\end{align}
Moreover, if $\dim\;(R/I) = i_0$ then by \cite[Corollary 4.5]{Tru01}
$$e_{j}\big(\mathbf{m}|I_{\geq\beta}\big)=\beta^je(R)\quad\mbox{for all } j=0,\ldots,d-i_0-1.$$
	
\begin{lemma}\label{lem1}
Let $R$ and $I$ be as in Notations \ref{n4} and assume that $\dim R/I=i_0$. Then
$$e_{i_0}\big(R[It]\big)=
e_{i_0+1}\big(R[It]\big)= \cdots = e_{d-2}\big(R[It]\big)=
0\quad\mbox{and}\quad e_{d-1}(R[It]) = e(R).$$
\end{lemma}
\begin{proof}
Since $R$ is a standard graded ring, it can be checked that $\sqrt{\langle I_{\beta}\rangle} = \sqrt{I}$. In particular,
\[\mathrm{height}\;\langle I_{\beta}\rangle = \mathrm{height}\;I = \dim R - \dim R/I = d-i_0.\]
Therefore, for all $j=1,\ldots,d-i_0-1$, we have
\begin{equation*}
e_{d-1-j}\big(R[It]\big)=\sum_{k=0}^{j}(-1)^k \binom{j}{k}\beta^{k}e_{j-k}\big(\mathbf{m}|\langle I_\beta \rangle\big)
=\beta^{j} e(R)\sum_{k=0}^j (-1)^k \binom{j}{k} = 
\beta^je(R)(1-1)^j = 0.
\end{equation*}
The second statement follows as $e(R) = H^{d-1} = e_{d-1}(R[It])$ (also by \cite[Theorem 4.2]{HT03}).
\end{proof}

\begin{rmk}\label{polarmultiplicity}
From the expressions \eqref{imprel1} and \eqref{imprel2}, it follows that the RA-multiplicities of $R[It]$ and $R[Jt]$ coincide if and only if the mixed multiplicities $\{e_i({\bf m}\vert I_{\geq\beta})\}_i$ and $\{e_i({\bf m}\vert J_{\geq\beta})\}_i$ coincide, for some integer $\beta \geq {\bf d}$. Whereas, it is known that the mixed multiplicities  $\{e_i({\bf m} \vert L)\}_i$ of an ideal $L$ are essentially  same as the polar multiplicities of $L$ due to Gaffney and Gassler \cite{GG99}. For an explanation of this statement, one may also refer to \cite[Remark 4]{AMP19} or \cite[Lemma 2.4]{Cid24}.

Therefore, the polar multiplicities of $I_{\geq\beta}$ and $J_{\geq\beta}$ are equal for some integer $\beta \geq {\bf d}$ if and only if the RA-multiplicities for $I$ and $J$ are equal. Similarly, the polar multiplicities of $\mathsf{I}_{\geq\beta}$ and $\mathsf{J}_{\geq\beta}$ are equal for some integer $\beta \geq {\bf d}$ if and only if the RA-multiplicities for $\mathsf{I}$ and $\mathsf{J}$ are equal.
\end{rmk}

\begin{defn}
Let $(R,{\bf m})$ be a $d$-dimensional Noetherian local ring and $I\subseteq R$ be an ideal. The \emph{$j$-multiplicity} of $I$ is defined to be the integer $$j(I) = \lim_{n\to\infty}\dfrac{\ell_R\left(H^0_{\bf m}\left(I^n/I^{n+1}\right)\right)}{n^{d-1}/(d-1)!}.$$
\end{defn}

In the following theorem, we obtain a localization-free criterion for integral dependence using various types of multiplicities.

\begin{thm}\label{comp}
Adopt Notations \ref{nnotation}. Further assume that $R$ is a domain. Then the following statements are equivalent:
\begin{enumerate}
\item[$(i)$] $\overline{I} = \overline{J}$.
\item[$(ii)$] $e\big(S[\mathsf{I}t]_{\Delta_{(c,1)}}\big) =  e\big(S[\mathsf{J}t]_{\Delta_{(c,1)}}\big)$ for some integer (all integers) $c>{\bf d}$.
\item[$(iii)$] $e_i\left(S[\mathsf{I}t]\right) = e_i\left(S[\mathsf{J}t]\right)$ for all $0\leq i\leq \dim R/I$.
\item[$(iv)$] $e_{i}\big(\mathbf{n} | \mathsf{I}_{\geq\bf d}\big)=e_{i}\big(\mathbf{n} | \mathsf{J}_{\geq\bf d}\big)$ for all $i$, where  $\mathrm{height}\;I\leq i\leq d$.
\item[$(v)$] $e_{d}\big(\mathbf{n} |\mathsf{I}_{\geq c}\big) = e_{d}\big(\mathbf{n} | \mathsf{J}_{\geq c}\big)$ for some integer (all integers) $c>{\bf d}$.
\item[$(vi)$] $j\big(\mathsf{I}_{\geq c}\big) = j\big(\mathsf{J}_{\geq c}\big)$ for some integer (all integers) $c>{\bf d}$.
\item[$(vii)$] $\varepsilon\big(\mathsf{I}_{\geq c}\big) = \varepsilon\big(\mathsf{J}_{\geq c}\big)$ for some integer (all integers) $c>{\bf d}$.
\end{enumerate}
\end{thm}
\begin{proof}
The assertions $(i)\iff (ii)\iff (iii)$ follow from Theorem \ref{integral1}   replacing $R$ by $S$ and ideals $I$ and $J$ by $\mathsf{I}$ and $\mathsf{J}$ respectively.

\vspace{5pt}

\noindent $(iii)\iff (iv)$. This is a direct consequence of the formulas \eqref{imprel1} and \eqref{imprel2} applied on $S$.

\vspace{5pt}

\noindent $(ii)\iff (v)$. Let $c>{\bf d}$ be an integer. The formulas \eqref{mm}, \eqref{imprel1} and \eqref{imprel2} applied on $S$, imply $$e\big(S[\mathsf{I}t]_{\Delta_{(c,1)}}\big) = \tfrac{1}{d+1}\cdot f_{\{\mathsf{I}^n\}}(c) = \sum_{i=0}^d \binom{d}{i}e_i\left(S[\mathsf{I}t]\right)c^i = e_{d}\big(\mathbf{n}| \mathsf{I}_{\geq c}\big),$$ and similarly $e\big(S[\mathsf{J}t]_{\Delta_{(c,1)}}\big) = e_{d}\big(\mathbf{n}|\mathsf{J}_{\geq c} \big)$. This proves the equivalence.

\vspace{5pt}

\noindent $(v)\iff (vi)$. As $c>{\bf d}$ so the ideals $\mathsf{I}_{\geq c}$ and $\mathsf{J}_{\geq c}$ are generated in equal degrees $c$. Also the ideals $\mathsf{I}_{\geq c}$ and $\mathsf{J}_{\geq c}$ have maximal analytic spreads because $e_{d}\big(\mathbf{n}| \mathsf{I}_{\geq c}\big) = e\big(S[\mathsf{I}t]_{\Delta_{(c,1)}}\big)>0$ and $e_{d}\big(\mathbf{n}| \mathsf{J}_{\geq c}\big) = e\big(S[\mathsf{J}t]_{\Delta_{(c,1)}}\big) >0$, see \cite[Corollary 3.6 and Corollary 3.7]{Tru01}. Then we have $j\big(\mathsf{I}_{\geq c}\big) = c\cdot e_{d}\big(\mathbf{n}|\mathsf{I}_{\geq c} \big)$ and $j\big(\mathsf{J}_{\geq c}\big) = c\cdot e_{d}\big(\mathbf{n}|\mathsf{J}_{\geq c}\big)$ by \cite[Corollary 3.5]{AMP19}. This proves the equivalence.

\vspace{5pt}

\noindent $(i)\iff (vii)$. Since the ideals $\mathsf{I}_{\geq c}$ and $\mathsf{J}_{\geq c}$ are generated in equal degrees $c$, we have
$$f_{\{(\mathsf{I}_{\geq c})^n\}}(x)
=  f_{\{(\mathsf{J}_{\geq c})^n\}}(x) = 0\quad\mbox{for all}\;\; x <c,$$
$$f_{\{(\mathsf{I}_{\geq c})^n\}}(x)
= f_{\{\widetilde{(\mathsf{I}_{\geq c})^n}\}}(x),
\quad\mbox{and}\quad 
 f_{\{(\mathsf{J}_{\geq c})^n\}}(x) =
 f_{\{\widetilde{(\mathsf{J}_{\geq c})^n}\}}(x)
 \quad\mbox{for all}\;\; x >c.$$
Therefore,
$$\varepsilon\big(\mathsf{J}_{\geq c}\big) - \varepsilon\big(\mathsf{I}_{\geq c}\big)
= \int_0^c \left(f_{\{\widetilde{(\mathsf{J}_{\geq c})^n}\}}(x)-f_{\{\widetilde{(\mathsf{I}_{\geq c})^n}\}}(x)\right)dx.$$
From Theorem \ref{mainsat} we know that $f_{\{\widetilde{(\mathsf{J}_{\geq c})^n}\}} - f_{\{\widetilde{(\mathsf{I}_{\geq c})^n}\}}$ is a nonnegative continuous function. Thus
$$\varepsilon\big(\mathsf{I}_{\geq c}\big) = \varepsilon\big(\mathsf{J}_{\geq c}\big)
\iff
f_{\{\widetilde{(\mathsf{I}_{\geq c})^n}\}} = f_{\{\widetilde{(\mathsf{J}_{\geq c})^n}\}}.$$
Also note that the quotients $\mathsf{I}/\mathsf{I}_{\geq c}$ and $\mathsf{J}/\mathsf{J}_{\geq c}$ are finite length $S$-modules. Therefore, $f_{\{\widetilde{(\mathsf{I}_{\geq c})^n}\}} = f_{\{\widetilde{\mathsf{I}^n}\}}$, and $f_{\{\widetilde{(\mathsf{J}_{\geq c})^n}\}}= f_{\{\widetilde{\mathsf{J}^n}\}}$.
Thus, $(i)\iff (vii)$ follows from Theorem \ref{adicfunction}.
\end{proof}

\subsection{Some examples}

\begin{rmk}\label{RAmultiplicity}
Following Notations \ref{n4}, if $\dim\;R/I = i_0$ then we have shown that the set of nonzero RA-multiplicities is a subset of $\{e_{0}(R[It]),\ldots, e_{i_0-1}(R[It]), e_{d-1}(R[It])\}$. However, there are examples which show that this inclusion can be proper.

It is not hard to see that equality of RA-multiplicities of $I$ and $J$ may not ensure $\overline{I} = \overline{J}$, {\em e.g.}, if $I =
{\bf m}^2\subset J= {\bf m}$, where $R$ is as before. Similarly the following example shows that only the equality of epsilon multiplicities is not sufficient to detect the equality of integral closures of homogeneous ideals.
\end{rmk}

\begin{ex}\label{elliptic}
Let $\mathbb{C}$ denote the field of complex numbers. Consider the ring $R=\mathbb{C}[X,Y,Z]/(X^3+Y^3+Z^3)=\mathbb{C}[x,y,z]$ with unique homogeneous maximal ideal ${\bf m}=(x,y,z)$. Choose prime ideals $\mathfrak{p}_1=(x+y,z)$ and $\mathfrak{p}_2=(x+z,y)$ associated to the points $[1:-1:0]$ and $[1:0:-1]$, respectively on the elliptic curve $\mathrm{Proj}\;R$. We set $J=\mathfrak{p}_1$ and $I=\mathfrak{p}_1 \cap \mathfrak{p}_2=(x+y+z,yz)$. Clearly, $I \subset J$ is not a reduction, as $\sqrt{I} \neq \sqrt{J}$.

We may use \cite[Remark 6.12]{DDRV24} to compute the $\varepsilon$-multiplicity of $J$, i.e., \[\varepsilon(J)=\dfrac{(e(R)-1)^2}{e(R)} = \dfrac{(3-1)^2}{3}=\dfrac{4}{3}.\]

In order to compute the $\varepsilon$-multiplicity of $I$ we introduce some notations. Let $S=R[w]$, where $w$ is a variable with $\deg w=1$, ${\bf n} = {\bf m}+(w)$ be the homogeneous maximal ideal of $S$, and $\mathsf{I} = IS$ be the extension of $I$ in $S$. Now from \cite[Theorem 6.10 or Discussions 6.21]{DDRV24}, we have
\[\varepsilon(I)= \dfrac{\left(e_1({\bf n}\vert \mathsf{I}_{\geq 2})\right)^2}{e_0({\bf n}\vert \mathsf{I}_{\geq 2})} - e_2({\bf n}\vert \mathsf{I}_{\geq 2}) = \dfrac{4^2}{3}-4 = \dfrac{4}{3}.\] Further, we get using {\sc Macaulay2} that
\begin{align*}
&e_0(\mathbf{m}| I_{\geq 2})= 3,~  e_1(\mathbf{m}| I_{\geq 2})=4, \quad \mbox{and} \quad e_0(\mathbf{m}| J)=3,~ e_1(\mathbf{m}| J)=2.
\end{align*}
Consequently, by \eqref{imprel1},
	\begin{equation*}
	\boxed{
		e(R)= 3, ~e_0\big(R[It]\big) =-2, \;\; \mbox{ and } \;\;  e_0\big(R[Jt]\big) =-1.
		}
	\end{equation*}
	Thus $\varepsilon(I)=\varepsilon(J)$ but $e_0\big(R[It]\big) \neq e_0\big(R[Jt]\big)$.
\end{ex}

\section{Reduction from domain to equidimensional}

\begin{notations}\label{n6}
Adopt Notations \ref{nnotation}. Here we allow $R$ to be equidimensional (not necessarily a domain). Let $\mathrm{Min}(A)$ denote the collection of minimal prime ideals of a Noetherian ring $A$.
\end{notations}

\begin{propose}\label{prop_associativity}
Let $c >d(I)$ be an integer. If ${\mathfrak p}$ be a minimal prime of $R$ such that  $I \not\subseteq {\mathfrak p}$ then $\tfrac{S[{\mathsf I}t]_{\Delta_{(c,1)}}}{\left(\oplus_{n\geq 0} (\mathfrak{p}S\cap {\mathsf I^n})_{cn}t^n\right)}$ is an integral domain of dimension  $d+1$. If $\mathrm{height}\;I>0$ then
\begin{equation}\label{minprimes}
\mathrm{Min}\big(S[{\mathsf It}]_{\Delta_{(c,1)}}\big) = \left\{\oplus_{n\geq 0} (\mathfrak{p}S\cap {\mathsf I^n})_{cn}t^n\mid \mathfrak{p} \in \mathrm{Min} (R)\right\}.
\end{equation}
In particular, $S[{\mathsf I}t]_{\Delta_{(c,1)}}$ is an equidimensional ring of dimension $d+1$.
\end{propose}
\begin{proof}
We note that $\mathrm{Min}(S) = \{\mathfrak{p}S \mid \mathfrak{p}\in \mathrm{Min}(R)\}$ and $\mathrm{Min}(S[t]) = \{\mathfrak{p}S[t] \mid \mathfrak{p}\in \mathrm{Min}(R)\}$. Therefore, both $S$ and $S[t]$ are equidimensional rings. Consider the composition of ring morphisms
$S[{\mathsf I}t]_{\Delta_{(c,1)}} \hookrightarrow S[{\mathsf I}t]\hookrightarrow S[t]$. Let ${\mathfrak p} \in \mathrm{Min}(R)$. Then the ideal
$$\mathfrak{p}S[t]\cap S[{\mathsf I}t]_{\Delta_{(c,1)}} = \oOplus_{n\geq 0} (\mathfrak{p}S\cap {\mathsf I}^n)_{cn}t^n$$
is a prime ideal and 
$$
\tfrac{S[{\mathsf I}t]_{\Delta_{(c,1)}}}{\mathfrak{p}S[t]\cap S[{\mathsf I}t]_{\Delta_{(c,1)}}}
\cong \oOplus_{n\geq 0} \big(\tfrac{{\mathsf I}^n}{\mathfrak{p}S\cap {\mathsf I}^n}\big)_{cn}t^n
 \cong \oOplus_{n\geq 0} \big(\tfrac{\mathsf{I}^n + \mathfrak{p}S}{\mathfrak{p}S}\big)_{cn}t^n
 \cong {\left(\tfrac{S}{\mathfrak{p}S}\big[\tfrac{{\mathsf I}+\mathfrak{p}S}{\mathfrak{p}S}t\big]\right)}_{\Delta_{(c,1)}}. $$
 
If $I\not\subseteq {\mathfrak p}$ then by \cite[Lemma 2.2]{HT03} or by \cite[Theorem 4.4]{DRT24}
$$\dim {\left(\tfrac{S}{\mathfrak{p}S}\big[\tfrac{{\mathsf I}+\mathfrak{p}S}{\mathfrak{p}S}t\big]\right)}_{\Delta_{(c,1)}} = \dim \tfrac{S}{\mathfrak{p}S}\big[\tfrac{{\mathsf I}+\mathfrak{p}S}{\mathfrak{p}S}t\big]-1 =
 \dim\tfrac{S}{{\mathfrak p}S} = \dim S = d+1.$$
 
Every minimal prime of $S[{\mathsf I}t]_{\Delta_{(c,1)}}$ is a contraction of some minimal prime of $S[t]$. Therefore, if $\mbox{height}~I>0$ then the equality \eqref{minprimes} and hence the rest of the assertion follows.
\end{proof}

\begin{rmk}
It is easy to see that the integral closure of a zero ideal is the nilradical ideal of the ring. In particular, the integral closure of any nilpotent ideal is the nilradical ideal of the ring. So to give a numerical characterization of the integral dependence, we can assume without any loss of generality that the ideal is non-nilpotent.
\end{rmk}

\begin{thm}\label{equidim}
Following Notations \ref{n6}, we further assume that $I$ is a  non-nilpotent ideal of $R$. Then the following statements are equivalent.
\begin{enumerate}
\item[$(i)$] $e\big(S[\mathsf{I}t]_{\Delta_{(c,1)}}\big) = e\big(S[\mathsf{J}t]_{\Delta_{(c,1)}}\big)$ for some (every) integer $c> {\bf d}$.
\item[$(ii)$] ${\overline I} = {\overline J}$.
\end{enumerate}
Further, if $\mathrm{height}\;I =\dim R$ then $e\big(S[\mathsf{I}t]_{\Delta_{(c,1)}}\big) = c^de({\bf m},R)-e(I,R)$ for every integer $c> {\bf d}$.
\end{thm}
\begin{proof}
Let $\mathrm{Min}(R) = \{{\mathfrak p}_1, \ldots, {\mathfrak p}_m\}$.

\vspace{5pt}
\noindent{\emph{Case $1$}}. Suppose $0 <\mathrm{height}\;I$.
Then the minimal prime ideals of $
S[\mathsf{I}t]_{\Delta_{(c,1)}}$  and $S[\mathsf{J}t]_{\Delta_{(c,1)}}$ are of the form 
$P_i = \oplus_{n\geq 0}(\mathfrak{p}_iS\cap \mathsf{I}^n)_{cn}t^n$ and 
$Q_i = \oplus_{n\geq 0}(\mathfrak{p}_iS\cap \mathsf{J}^n)_{cn}t^n$ respectively.
Further, $Q_i = Q_j$ implies that $P_i = P_j$. Therefore, relabelling if necessary, we can write
$\{P_1, \ldots, P_r\}$ and $\{Q_1, \ldots, Q_s\}$ as the set of all distinct minimal primes of 
$S[\mathsf{I}t]_{\Delta_{(c,1)}}$ and $S[\mathsf{J}t]_{\Delta_{(c,1)}}$ respectively.
In particular, $r\leq s \leq m$.

For every minimal prime ${\mathfrak p}$ of $R$, the natural homomorphism $$\varphi_{\mathfrak{p}}\colon {S[\mathsf{I}t]}_{\mathfrak{p}S[t]\cap S[\mathsf{I}t]} \longrightarrow {S[t]}_{\mathfrak{p}S[t]}$$ of Artinian local rings is an isomorphism.  This induces an isomorphism on the diagonal subalgebras, i.e.,
$$\left(S[\mathsf{I}t]_{\Delta_{(c,1)}}\right)_{\mathfrak{p}S[t]\cap S[\mathsf{I}t]_{\Delta_{(c,1)}}} \cong \left(S[t]_{\Delta_{(c,1)}}\right)_{\mathfrak{p}S[t]\cap S[t]_{\Delta_{(c,1)}}}.$$

Now the associativity formula for Hilbert-Samuel multiplicities \cite[Corollary 4.7.8]{BH98} gives
\begin{align}\label{assoc-formula1}
 e\left(S[\mathsf{I}t]_{\Delta_{(c,1)}}\right) &= \sum_{i=1}^r \ell\Big({\left(S[\mathsf{I}t]_{\Delta_{(c,1)}}\right)}_{P_i}\Big)\cdot e\Big(\tfrac{S[\mathsf{I}t]_{\Delta_{(c,1)}}}{P_i}\Big) \nonumber\\
 &= \sum_{i=1}^r \ell\Big(\big(S[t]_{\Delta_{(c,1)}}\big)_{\mathfrak{p}_iS[t]\cap S[t]_{\Delta_{(c,1)}}}\Big)\cdot
 e\Big({\tfrac{S}{\mathfrak{p}_iS}\big[\tfrac{\mathsf{I}+\mathfrak{p}_iS}{\mathfrak{p}_iS}t\big]}_{\Delta_{(c,1)}}\Big).
\end{align}

Similarly we can write 
\begin{equation}\label{assoc-formula2}
 e\left(S[\mathsf{J}t]_{\Delta_{(c,1)}}\right) =  \sum_{i=1}^s \ell\Big(\big(S[t]_{\Delta_{(c,1)}}\big)_{\mathfrak{p}_iS[t]\cap S[t]_{\Delta_{(c,1)}}}\Big)\cdot e\Big({\tfrac{S}{\mathfrak{p}_iS}\big[\tfrac{\mathsf{J}+\mathfrak{p}_iS}{\mathfrak{p}_iS}t\big]}_{\Delta_{(c,1)}}\Big).
\end{equation}

Since $\mathsf{I}\subseteq \mathsf{J}$, for every such ${\mathfrak{p}_i}$ we have
\begin{equation}\label{eqminprime}
e\Big({\tfrac{S}{\mathfrak{p}_iS}\big[\tfrac{\mathsf{I}+\mathfrak{p}_iS}{\mathfrak{p}_iS}t\big]}_{\Delta_{(c,1)}}\Big) \leq e\Big({\tfrac{S}{\mathfrak{p}_iS}\big[\tfrac{\mathsf{J}+\mathfrak{p}_iS}{\mathfrak{p}_iS}t\big]}_{\Delta_{(c,1)}}\Big).
\end{equation}

Therefore, the equality $e\big(S[\mathsf{I}t]_{\Delta_{(c,1)}}\big) = e\big(S[\mathsf{J}t]_{\Delta_{(c,1)}}\big)$
holds if and only if $r=s$ and the inequality in \eqref{eqminprime} is an equality for all $i$. On the other hand by \cite[Proposition 1.1.5]{HS06} we know that
$$\overline{I} = \overline{J}\quad \text{if and only if} \quad \overline{(I+\mathfrak{p}_i)/\mathfrak{p}_i} = \overline{(J+\mathfrak{p}_i)/\mathfrak{p}_i} \quad \text{for all}~i=1,\ldots,m.$$
By Theorem \ref{comp},
$\overline{(I+\mathfrak{p}_i)/\mathfrak{p}_i} = \overline{(J+\mathfrak{p}_i)/\mathfrak{p}_i}$
if and only if for some (every) integer $c > {\bf d}$, we have
$$e\Big({\tfrac{S}{\mathfrak{p}_iS}
\big[\tfrac{{\mathsf{I}+\mathfrak{p}_iS}}{\mathfrak{p}_iS}\big]}_{\Delta_{(c,1)}}\Big)
= e\Big({\tfrac{S}{\mathfrak{p}_iS}\big[\tfrac{\mathsf{J}+\mathfrak{p}_iS}{\mathfrak{p}_iS}\big]}_{\Delta_{(c,1)}}\Big).$$

\vspace{5pt}
\noindent{\emph{Case $2$}}.\quad Suppose that $\mbox{height}~I =0$. Even then $\dim S[\mathsf{I}t]_{\Delta_{(c,1)}} =d+1$ as there exists at least one  minimal prime of $R$ which  does not contain $I$.
On the other hand if $I\subseteq {\mathfrak{p}_i}$ then  
${\frac{S}{\mathfrak{p}_iS}\big[\tfrac{\mathsf{I}+\mathfrak{p}_iS}{{\mathfrak{p}_iS}}t\big]}_{\Delta_{(c,1)}} = R_0$ and those $\mathfrak{p}_i$ do not contribute to the multiplicity of
$S[\mathsf{I}t]_{\Delta_{(c,1)}}$.

Now by relabelling if necessary we can assume that $P_1, \ldots, P_{r_1}$ are all the minimal primes of $S[\mathsf{I}t]_{\Delta_{(c,1)}}$ such that $I\not\subseteq {\mathfrak{p}_i}$, for $1\leq i\leq r_1$.
Similarly we can assume that $Q_1, \ldots, Q_{s_1}$ are all the minimal primes of $S[\mathsf{J}t]_{\Delta_{(c,1)}}$ such that $J\not\subseteq {\mathfrak{p}_i}$ for $1\leq i\leq s_1$. It is obvious that $r_1\leq s_1$.

Now replacing $r$ by $r_1$ in \eqref{assoc-formula1} and $s$ by $s_1$ \eqref{assoc-formula2}, and arguing as before we conclude that the equality $e\big(S[\mathsf{I}t]_{\Delta_{(c,1)}}\big)=e\big(S[\mathsf{J}t]_{\Delta_{(c,1)}}\big)$ holds if and only if
$\overline{(I+\mathfrak{p}_i)/\mathfrak{p}_i} = \overline{(J+\mathfrak{p}_i)/\mathfrak{p}_i}$ for all $1\leq i\leq r_1$. For $r_1+1\leq i\leq m$ we have $\overline{(I+\mathfrak{p}_i)/\mathfrak{p}_i} = \overline{(J+\mathfrak{p}_i)/\mathfrak{p}_i} = 0$. This proves that the assertions $(i)$ and $(ii)$ are  equivalent.

\vspace{5pt}

Suppose that $\mbox{height}\;I = \dim~R$.
Then for all $n\gg 0$, we have $R_{cn+1} = (I^n)_{cn+1}$ which implies
$$e(I,R) = \lim_{n\to 0} \frac{\sum_{m=0}^{cn}\ell_{R_0}(R_m)}{n^d/d!} -
\lim_{n\to 0} \frac{\sum_{m=0}^{cn}\ell_{R_0}\left({(I^n)}_m\right)}{n^d/d!} = c^de({\bf m},R)-e\big(S[\mathsf{I}t]_{\Delta_{(c,1)}}\big).$$
Our assertion now follows from Rees' theorem \cite{Ree61}.

\end{proof}

\section{A Macaulay2 algorithm}\label{M2section}
The Macaulay2 package \href{https://macaulay2.com/doc/Macaulay2/share/doc/Macaulay2/MultiplicitySequence/html/index.html}{\sf MultiplicitySequence}, authored by Chen, Kim and Monta\~no \cite{CKM24}, contains the commands `multiplicitySequence', and `jMult' to compute the multiplicity sequence and $j$-multiplicity of an ideal respectively. So one might be interested to know if there is a Macaulay2 algorithm to compute the density functions at certain slopes to detect if a homogeneous ideal contained in another is a reduction. In this section, we present such an algorithm based on Theorem \ref{comp}.

\vspace{0.15cm}
\noindent
{\bf Algorithm}. Adopt Notations \ref{nnotation}. Further assume that $R$ is a domain. Set $c={\bf d}+1$.

\vspace{0.1cm}
\noindent
{\bf Input}. The sequence $W = (J, I)$.

\vspace{0.1cm}
\noindent
{\bf Output}. Script-1, which is based on Theorem \ref{comp} $(ii)$, detects if $e\big(S[\mathsf{I}t]_{\Delta_{(c,1)}}\big)$ and  $e\big(S[\mathsf{J}t]_{\Delta_{(c,1)}}\big)$ are identical or not and gives an output in the form \emph{true/false}. Whereas, Script-2 is based on Theorem \ref{comp} $(iv)$ and detects whether all the (Teissier) mixed multiplicities are equal. It essentially detects if $e_j\big(\mathsf{n} | \mathsf{I}_{\geq \bf d}\big)=e_j\big(\mathsf{n} | \mathsf{J}_{\geq \bf d}\big)$ for $j=\mbox{height}\;I, \ldots, d$ in $S$. Based on Theorem \ref{comp} $(v)$, one only needs to check whether the top (Teissier) mixed multiplicities are equal, i.e., if $e_d\big(\mathsf{n} | \mathsf{I}_{\geq c}\big)=e_d\big(\mathsf{n} | \mathsf{J}_{\geq c}\big)$ in $S$. However, as the number of generators of $\mathsf{I}_{\geq c}$ and $\mathsf{J}_{\geq c}$ are more than those of $\mathsf{I}_{\geq {\bf d}}$ and $\mathsf{J}_{\geq {\bf d}}$ so using Theorem \ref{comp} $(v)$ may not be computationally efficient.

\vspace{0.25cm}
\noindent
{\bf Script 1}.
\begin{multicols}{2}
	{\footnotesize
		\begin{verbatim}
		--W=(J,I)
		reductionMultDiagonal=W->(
		J:=W#0;
		I:=W#1;
		dJ:=(max degrees J)#0;
		dI:=(max degrees I)#0;
		c:=max{dJ,dI}+1;
		R:=ring J;
		T := getSymbol"T";
		S1:=(flattenRing(R[T]))#0;
		S:=newRing(S1,Degrees=>{numgens S1:1});
		JS:=sub(J,S);
		IS:=sub(I,S);
		n:=ideal vars S;
		JSc:=intersect(JS,n^c);
		ISc:=intersect(IS,n^c);
		KJ:=specialFiberIdeal(JSc, JSc_0);
		KI:=specialFiberIdeal(ISc, ISc_0);
		eI:= degree KI;
		eJ:= degree KJ;
		eI==eJ
		)
		\end{verbatim}
	}
\end{multicols}

\noindent
{\bf Script 2}.
\begin{multicols}{2}
	{\footnotesize
		\begin{verbatim}
		--needsPackage "MixedMultiplicity"
		--W=(J,I)
		reductionMixedMult=W->(
		J:=W#0;
		I:=W#1;
		dJ:=(max degrees J)#0;
		dI:=(max degrees I)#0;
		d:=max{dJ,dI};
		R:=ring J;
		dR:=dim R;
		htI:=dim R - dim (R/I);
		T := getSymbol"T";
		S1:=(flattenRing(R[T]))#0;
		S:=newRing(S1,Degrees=>{numgens S1:1});
		n:=ideal vars S;
		JS:=sub(J,S);
		IS:=sub(I,S);
		JSd:=intersect(JS,n^d);
		ISd:=intersect(IS,n^d);
		for i from htI to dR do << i << "->"
		<<mixedMultiplicity ((n,ISd),(dR-i, i))==
		mixedMultiplicity ((n,JSd),(dR-i, i))<< endl;)
		\end{verbatim}
	}
\end{multicols}

\begin{ex}\label{CRPU}\cite[Example 6.9]{CRPU24}
Let $R=k[X,Y]$ be a polynomial ring over a field $k$ and consider the ideals $I=(X^2,XY^2) \subseteq J=(X^2, XY)$. We can write
	\[I^n=(X)^n \cdot (X,Y^2)^n \quad \mbox{and}\quad J^n=(X)^n \cdot (X,Y)^n\]
	for every $n \geq 1$. Using arguments as in \cite[Example 5.10]{DDRV24}, we get that
	\[\varepsilon(I)=e\big((X,Y^2)\big)=\ell_R(R/(X,Y^2))=2 \quad \mbox{and} \quad \varepsilon(J)=e\big((X,Y)\big)=\ell_R(R/(X,Y))=1.\]
	So $I \subseteq J$ is not a reduction.
\end{ex}
To verify this example, we run the following session in Macaulay2.
\begin{multicols}{2}
	{\small
		\begin{verbatim}
		i1 : R=QQ[X,Y];
		i2 : I=ideal(X^2,X*Y^2);
		i3 : J=ideal(X^2, X*Y);
		i4 : isSubset(I,J)
		o4 =  true
		i5 : time reductionMultDiagonal (J,I)
		-- used 0.0450719s
		o5 = false
		o5 :  Sequence
		i6 : needsPackage "MultiplicitySequence"
		i7 : time multiplicitySequence J
		-- used 0.106662s
		o7 =  HashTable{1 => 1, 2 => 2}
		i8 :  time multiplicitySequence I
		-- used 0.0885962s
		o8 =  HashTable{1 => 1, 2 => 4}
		i9 : needsPackage "ReesAlgebra"
		i10 : time isReduction (J,I)
		-- used 0.00602714s
		o10 =  false
		i11 : needsPackage "MixedMultiplicity"
		i12 : time reductionMixedMult (J,I)
		1->true
		2->false
		-- used 0.0744045s
		\end{verbatim}
	}
\end{multicols}

For the above computations, both the scripts seem more efficient than `MultiplicitySequence'.

\section{Acknowledments}
The authors would like to thank Prof. Steven Dale Cutkosky for several useful discussions concerning Proposition \ref{impsat1}. They would also like to thank Prof. Ngo Viet Trung for suggesting an improvement to Theorem \ref{comp} which simplified the algorithm in Section \ref{M2section}. The second author  would like to thank DAE for providing financial support during her tenure as a visiting fellow at TIFR, Mumbai, where this project was started. After that she was supported by the CRG grant (CRG/2022/007572) of SERB when she was a research associate at IIT Dharwad.

\bibliographystyle{alpha}
\bibliography{Ref}
\end{document}